\documentclass[11pt]{article}

\usepackage{amssymb}
\usepackage{amsmath}
\usepackage{amsthm}

\usepackage{subfigure}
\usepackage{float}
 \usepackage{paralist}
 \usepackage{graphics} 
  \usepackage{epsfig} 
\usepackage{graphicx}

\usepackage{hyperref}
\usepackage{geometry}
\newtheorem{theorem}{Theorem}[section]

\newtheorem{lemma}[theorem]{Lemma}
\newtheorem{sublem}[theorem]{Sublemma}
\newtheorem{proposition}[theorem]{Proposition}
\theoremstyle{definition}
\newtheorem{definition}{Definition}
\newtheorem{remark}[theorem]{Remark}
\newtheorem{example}{Example}

\begin{document}

\title{Existence of  torsion-low maximal  isotopies for  area preserving surface homeomorphisms}
\author{Jingzhi  YAN \footnote{
jyan@scu.edu.cn  
}}

\maketitle

\begin{abstract}
The  paper concerns  area preserving homeomorphisms  of surfaces  that are isotopic to the identity. The purpose of the paper is to find a maximal  isotopy  such that we can give a fine description of the dynamics of its transverse foliation. We will define a sort of identity isotopies: torsion-low isotopies. In particular, when $f$ is a diffeomorphism with finitely many fixed points such that every fixed point is not degenerate,   an identity isotopy $I$ of $f$ is torsion-low if and only if for every  point $z$ fixed along the isotopy, the (real) rotation number $\rho(I,z)$ (which is well defined when one blows up $f$ at $z$) is contained in $(-1,1)$. We will prove the existence of torsion-low maximal  isotopies, and  will deduce the local dynamics of the transverse foliations of any torsion-low maximal isotopy near any isolated singularity.
\end{abstract}

\smallskip
\noindent \textbf{Keywords.} Surface homeomorphism, Transverse foliation, Maximal isotopy, Rotation number, Local rotation set.

\section{Introduction and definitions}
 In this article, we will study maximal isotopies and transverse foliations for  homeomorphisms of oriented surfaces.  The objects are fruitful tools in the study of homeomorphisms of surfaces. For example, one can prove the existence of periodic orbits in several cases \cite{lecalvezfeuilletage}, \cite{lecalvezcreateperiodicorbit}, \cite{Yan};  one can give  precise descriptions  of the dynamics of some  homeomorphisms  of the plane \cite{lerouxrotation}, of the annulus \cite{Matsumotoannulushomeo}, of the torus $\mathbb{T}^2$ \cite{davalos}, \cite{Davalos2},   \cite{Andres},\cite{korepecki},  and also of the general compact oriented surface  \cite{LeCalvezfocingtheory}; \ldots.

More precisely, an \emph{identity isotopy} $I$ of a surface homeomorphism $f$ is a continuous family of homeomorphisms $(f_t)_{t\in[0,1]}$  with $f_0=\mathrm{Id}$ and $f_1=f$, and a \emph{maximal isotopy} is an identity isotopy $I=(f_t)_{t\in[0,1]}$ of $f$ such that $f$ does not have any fixed point whose trajectory along $I$ is contractible in $M\setminus\mathrm{Fix}(I)$, where  $\mathrm{Fix}(I)=\cap_{t\in[0,1]}\mathrm{Fix}(f_t)$ is the fixed points set of $I$. To such an isotopy, one can associate \emph{transverse foliations} \cite{lecalvezfeuilletage}, i.e. oriented singular foliations $\mathcal{F}$ whose singular points set coincides with the fixed points set of $I$,  such that the trajectory of each point $z\notin \mathrm{Fix}(I)$  is homotopic (relatively to the end points) to a path that is positively transverse to the foliation (that means the path locally crosses every leaf from the left to the right). The dynamics of the maximal isotopy and the transverse foliation are ``dual" to each other. One example is the  isotopy  defined by the flow induced by a vector field $X$, and  the  foliation whose leaves are the integral curves of a vector field $Y$  transverse to $X$. In particular, we suppress  $X$  a Hamiltonian vector field and $Y$   the gradient vector field of the Hamiltonian function, so we can view the maximal isotopy and the transverse foliation as a kind of generalizations of the Hamiltonian flow and   a gradient flow. Indeed, if $f$ is defined by a (time dependent) Hamiltonian isotopy and $I$ is a maximal isotopy, then one can prove that the dynamics of a transverse foliation look like the dynamics of a gradient flow \cite{lecalvezfeuilletage}.

 A maximal isotopy always exists \cite{BCLRisotopies}, but is not unique.
A natural question is whether there exist maximal isotopies ``better" than the others. In the example of last paragraph,  the index of the foliation at each isolated singular point coincides with the Lefschetz index of the homeomorphism at the same point. Is there a class of maximal isotopies with this property in more general cases? (We will see that the torsion-low maximal isotopies in this article keep this property.)

Another observation is the following: among all the maximal isotopies, there may exist some  maximal isotopies that fix more points than the others. So we want to find some criteria for maximal isotopies  to fix more points.
When $f:M\rightarrow M$ is an area preserving diffeomorphism, we have a criterion based on the blow-up rotation number. More precisely, at each fixed point $z_0$ of $I$, we can give a natural  blow-up at $z_0$ by replacing $z_0$ with the unit circle of the tangent space $U_{z_0}M$. The extension of $f$ to this circle can be induced by the derivative $Df(z_0)$. We can define a \emph{blow-up rotation number} $\rho(I,z_0)\in\mathbb{R}$, which is a representative of the Poincar\'e's   rotation number of  the homeomorphism on the circle added (see Section \ref{S: pre-blow up}). Moreover,  if there exists a fixed point $z_0\in\mathrm{Fix}(I)$ such that $|\rho(I,z_0)|>1$ and that the connected component $M_0$ of $M\setminus(\mathrm{Fix}(I)\setminus\{z_0\})$ containing $z_0$ is not homeomorphic to  a sphere or a plane, then we can find another fixed point of $f$ that is not a fixed point of $I$ as a corollary of a generalized version of Poincar\'e-Birkhoff theorem.
 Let us explain briefly the reason:   in this case, $z_0$ is  isolated  in $\mathrm{Fix}(I)$. We consider the universal cover $\pi:\widetilde{M}\rightarrow M_0$ and the lift $\widetilde{f}$ of $f|_{M_0}$ to $\widetilde{M}$ that fixes every point in $\pi^{-1}\{z_0\}$. Fix $\widetilde{z}_0\in\pi^{-1}(z_0)$ and  consider the blow-up of $\widetilde{f}$ at $\widetilde{z}_0$. One gets a homeomorphism of the annulus $(\widetilde{M}\setminus\{\widetilde{z}_0\})\sqcup U_{z_0}\widetilde{M}$. By a generalized version of Poincar\'e-Birkhoff theorem, this homeomorphism has a fixed point $\widetilde{z}$ such that $\pi(\widetilde{z})$ is a fixed point of $f$ but is not a fixed point of $I$. Moreover, if $\mathrm{Fix}(I)$ is finite, by a technical discussion,  one can prove the existence  of another maximal isotopy that fixes $\mathrm{Fix}(I)\setminus\{z_0\}$ and has no less (probably more) fixed points than $I$ (see Section \ref{S: existence of a  torsion-low maximal isotopy}).  Then, it is reasonable to think that a maximal  isotopy  $I$ such that
\[-1\le \rho(I,z)\le 1 \textrm{ for all } z\in\mathrm{Fix}(I),\]
fixes more fixed points than a usual one\footnote{For a generic diffeomorphism (all the fixed points are non degenerate), both inequations can be  strict.}. In this article, we will see that torsion-low maximal isotopies satisfy the inequalities.

\bigskip

 Now, we will define   torsion-low maximal isotopies and give an exact description about what we will do in this article.

Let $M$ be a connected and oriented surface. We  write $f:(W,z_0)\rightarrow (W',z_0)$ for an orientation preserving homeomorphism between two neighborhoods $W$ and $W'$ of $z_0\in M$ such that $f(z_0)=z_0$. Such a local homeomorphism $f$ is called an \emph{orientation preserving local homeomorphism at $z_0$}. A \emph{local isotopy} $I$ of $f$  is a continuous family of local homeomorphisms $(f_t)_{t\in[0,1]}$ fixing $z_0$ (see p.146 of \cite{lecalveztourner}  for a more precise definition).  We say that $z$ is a \emph{contractible} fixed point of $f$ associated to the local isotopy $I$ if the trajectory $t\mapsto f_t(z)$  of $z$ along $I$ is a loop homotopic to zero in $W\setminus\{z_0\}$.
Let $\mathcal{F}$ be a singular oriented foliation on $M$. We say that $\mathcal{F}$  is \emph{locally transverse} to a local isotopy $I=(f_t)_{t\in[0,1]}$ at $z_0$, if there exists a neighborhood $U_0$ of $z_0$ such that $\mathcal{F}|_{U_0}$ has exactly one singularity $z_0$, and if for every  sufficiently small neighborhood $U$ of $z_0$, there exists a neighborhood $V\subset U$ of $z_0$ such that for all $z\in V\setminus \{z_0\}$, the trajectory $t\mapsto f_t(z)$  of $z$ along  $I$ is homotopic (relatively to the end points) in $U\setminus\{z_0\}$  to a path that is  positively transverse to $\mathcal{F}$.

Let  $f: (W,z_0)\rightarrow (W',z_0)$ be an orientation  preserving local homeomorphism at $z_0$. We will generalize the definitions of ``positive type'' and ``negative type'' by Shigenori Matsumoto \cite{Matsumoto}. We say that $I$  has a \emph{positive} (resp. \emph{negative} or \emph{zero}) \emph{rotation type} at $z_0$ if there exists a  foliation $\mathcal{F}$  locally transverse to $I$ such that $z_0$ is a sink (resp. source or saddle) of $\mathcal{F}$ \footnote{The precise definition of a sink, a source and a saddle will be given in Section \ref{S: pre-local dynamics of transverse foliation}.}.  Two local isotopies $I$ and $I'$ have the same rotation type if they are locally homotopic. Moreover, when $f$ is area preserving and $z_0$ is an isolated fixed point, we can prove that a local isotopy of $f$ has exactly one of the previous rotation types (see Section \ref{S: rotation type at an isolated fixed point}).

Recall that $\pi_1(\mathrm{homeo}_0(\mathbb{R}^2,0))\cong \mathbb{Z}$, where $\mathrm{homeo}_0(\mathbb{R}^2,0)$ is the space of homeomorphisms of $\mathbb{R}^2$ fixing $0$ and isotopic to the identity (see \cite{McCarty} or \cite{Hanmstrom-Homeotopygroups}). So,  we can  define a preorder on the set of all local isotopies of $f$ such that $I\lesssim I'$ if and only if there exists $k\ge 0$ such that $I'$ is locally homotopic to $J_{z_0}^kI$, where $J_{z_0}=(R_{2\pi t})_{t\in[0,1]}$ is the local isotopy of the identity and each $R_{2\pi t}$ is the counter-clockwise rotation through an angle $2\pi t$ about the center $z_0$.

\begin{definition}\label{Def: torsion-low-isolated}
 We  say that a local isotopy $I$ of an orientation and area preserving local homeomorphism $f$ at an isolated fixed point  is \emph{torsion-low} if
 \begin{itemize}
 \item[-] every local isotopy $I'>I$ has a positive rotation type;
\item[-] every local isotopy $I'<I$ has a negative rotation type.
\end{itemize}
\end{definition}
We can show
 that such a local isotopy always exists. Formally, we have the following result:

\begin{proposition}\label{P: local rotation type}
Let $f: (W, z_0)\rightarrow (W',z_0)$ be an orientation and area preserving local homeomorphism with an isolated fixed point $z_0$.   Then,
\begin{itemize}
\item[-] a local isotopy of $f$  has exactly one of the three kinds of rotation  types;
 \item[-]  there exists a local isotopy $I_0$ that is torsion-low at $z_0$. Moreover, $I_0$  has a zero rotation type if the Lefschetz index $i(f,z_0)$ is different from $1$, and has  either a positive or a negative rotation type if the Lefschetz index $i(f,z_0)$ is equal to $1$.
\end{itemize}
\end{proposition}

When $z_0$ is  a non-isolated fixed point of a local homeomorphism $f: (W,z_0)\to (W', z_0)$, one may fail to find a locally transverse foliation $\mathcal{F}$ for a local isotopy $I$ of $f$ at $z_0$. Still, we  get the following proposition that generalize Proposition \ref{P: local rotation type}.

\begin{proposition}\label{P: local rotation type non-isolated}
Let $f: (W, z_0)\rightarrow (W',z_0)$ be an orientation and area preserving local homeomorphism with a non-isolated fixed point $z_0$, and $I$ a local isotopy of $f$. Then,  exactly one of the following three situations occurs:
\begin{itemize}
\item[i)]  $z_0$ is accumulated by contractible fixed points of $f$ associated to $I$,
\item[ii)] $I$ has a positive rotation type,
\item[iii)] $I$ has a negative rotation type.
\end{itemize}
Moreover,
\begin{itemize}
\item[-]if $I$ has a positive (resp. negative) rotation type, every local isotopy $I'>I$ (resp. $I'<I$) has a positive (resp. negative) rotation type;
\item[-]if $z_0$ is accumulated by contractible fixed points of $f$ associated to $I$, for every local isotopy $I'>I$, either situation i) or situation ii) occur, and for every local isotopy $I'<I$, either situation i) or situation iii) occur.
 \end{itemize}
\end{proposition}

 We will generalize the definition of torsion-low isotopy as follows:

\begin{definition}\label{Def: torsion-low-nonsolated}
We  say that a local isotopy  $I$ of an orientation and area preserving local homeomorphism $f$ at $z_0$ is \emph{torsion-low} if
\begin{itemize}
\item[-] for every local isotopy $I'>I$, either $I'$ has a positive rotation type, or $z_0$ is accumulated by contractible fixed points of $f$ associated to $I'$;
\item[-] for every local isotopy $I'<I$, either $I'$ has a negative rotation type, or $z_0$ is accumulated by contractible fixed points of $f$ associated to $I'$.
\end{itemize}
\end{definition}

\begin{remark}
When $z_0$ is an isolated fixed point, the definition coincides with Definition \ref{Def: torsion-low-isolated}.
\end{remark}

\begin{remark}\label{R: ac means torsion-low}
If $z_0$ is accumulated by contractible fixed points of $f$ associated to $I$, then $I$, $J_{z_0}I$, $J_{z_0}^{-1}I$ are torsion-low,  where $J_{z_0}=(R_{2\pi t})_{t\in[0,1]}$ is the local isotopy of the identity and each $R_{2\pi t}$ is the counter-clockwise rotation through an angle $2\pi t$ about the center $z_0$.
\end{remark}

\medskip

When $f$ is a diffeomorphism  fixing $z_0$, and $I$ is a local isotopy of $f$, we can blow-up $f$ at $z_0$ and  define the blow up rotation number $\rho(I,z_0)$. We say that $z_0$ is a \emph{degenerate} fixed point of $f$ if $1$ is an eigenvalue of  $Df(z_0)$.  When $f$ is a homeomorphism, one may fail to find a blow-up at $z_0$, and cannot define a rotation \textquotedblleft number". However, we can  generalize it and define a \emph{local rotation set} $\rho_s(I, z_0)$ which was introduced by Le Roux \cite{lerouxrotation} and will be recalled  in Section \ref{S: pre-local rotation set}.  A torsion-low local isotopy has the following properties:
\begin{proposition}\label{P: rotation set of a  torsion-low local isotopy}
Let $f: (W,z_0)\rightarrow (W',z_0)$ be an orientation and area  preserving homeomorphism, and $I$ a torsion-low local isotopy of $f$. Then,
\[\rho_s(I,z_0)\cap [-1,1]\ne \emptyset.\]
In particular, if $z_0$ is an isolated fixed point of $f$,
\[\rho_s(I,z_0)\subset [-1,1].\]
If $f$ can be blown up at $z_0$, the rotation set is reduced to  a single point in $[-1,1]$.
Moreover, if  $f$ is  a diffeomorphism in a neighborhood of $z_0$, the blow-up rotation number satisfies
\[-1\le \rho(I,z_0)\le 1,\]
and  the inequalities are both strict when $z_0$ is not degenerate.
\end{proposition}

 When $z_0$ is a non-isolated fixed point, one may fail to find a torsion-low local isotopy in some particular cases. In fact, there exists  an orientation and area  preserving local homeomorphism whose local rotation set is reduced to $\infty$, and hence by Proposition \ref{P: rotation set of a  torsion-low local isotopy} there does not exist any torsion-low isotopy of this local homeomorphism (see Example \ref{Ex: local rotation set is infty}).

\medskip

However, when $f$ is an  area preserving homeomorphism of $M$ that  is isotopic to the identity, we can find a maximal  isotopy $I$ of $f$ that is torsion-low (as a local isotopy) at every fixed point of $I$. We will call such an isotopy a torsion-low maximal isotopy, and will prove the existence of a torsion-low maximal isotopy, which is the main result of this article, in Section \ref{S: existence of a  torsion-low maximal isotopy}.  More precisely, we have the following definition and theorem.
\begin{definition}
An identity isotopy of $f$ is  \emph{torsion-low} if it is torsion-low as a local isotopy  at each of its fixed points.
\end{definition}

\begin{theorem}\label{T: main}
Let $f$ be an  area preserving  homeomorphism of  $M$ that is isotopic to the identity. Then, there exists a  torsion-low maximal  isotopy $I$ of $f$.
\end{theorem}

\begin{remark}\label{R: area preserving is neccesary for the existence of globle torsion-low isotopy}
The area preserving condition is necessary for the result of this theorem. Even if $f$ has only finitely many fixed points and is area preserving near each fixed point,  one may still fail to find a   maximal isotopy $I$ that is torsion-low at every $z\in\mathrm{Fix}(I)$ (see Example \ref{Ex: area preserving condition is necessay}).
\end{remark}

 A torsion-low maximal isotopy gives more information than a usual one. We have  the following  three results related to the questions at the beginning of this section.  The first one will be proved in Section \ref{S: Local dynamics  of the locally transverse foliation of a torsion-low local isotopy}, the second will be proved in the end of Section \ref{S: existence of a  torsion-low maximal isotopy}, and  the third one is an immediately corollary of Proposition \ref{P: rotation set of a  torsion-low local isotopy} and Theorem \ref{T: main}.

 \begin{proposition}\label{P: local dynamic of the transverse foliation of a torsion-low isotopy}
Let $f$ be an  area preserving homeomorphism of  $M$ that is isotopic to the identity,  $I$  a maximal  isotopy of $f$ that is torsion-low  at  $z\in\mathrm{Fix}(I)$, and $\mathcal{F}$  a transverse foliation of $I$. If $z$ is an isolated singularity of  $\mathcal{F}$, then we have the following results:
\begin{itemize}
\item[-]if $z$ is an isolated fixed point of $f$ such that the Lefschetz index $i(f,z)\ne 1$, then $z$ is a saddle of $\mathcal{F}$ and $i(\mathcal{F},z)=i(f,z)$;
\item[-]if $z$ is an isolated fixed point of $f$ such that the Lefschetz index $i(f,z)=1$, or if $z$ is not isolated in $\mathrm{Fix}(f)$, then  $z$ is a sink or a source of $\mathcal{F}$.
\end{itemize}
 \end{proposition}

\begin{proposition}\label{P: torsion-low isotopy has most fixed points}
Let $f$ be an area preserving homeomorphism of $M$ that is isotopic to the identity and has finitely many fixed points.
Let
\[n=\max\{\#\mathrm{Fix}(I): I \text{ is an identity isotopy of } f\}.\]
Then there exists a torsion-low maximal isotopy of $f$ with $n$ fixed points.
\end{proposition}

\begin{proposition}\label{P: diffeo}
Let  $f$ be an  area preserving diffeomorphism of  $M$ that is isotopic to the identity. Then,  there exists a maximal  isotopy $I$ of $f$, such that for all $z\in\mathrm{Fix}(I)$,
\[-1\le \rho(I,z)\le 1.\]
Moreover, both inequalities are  strict when $z$ is not  degenerate.
\end{proposition}

\begin{remark}\label{R: inegality is not strict}
One may fail to get the strict inequalities without the assumption of non-degeneracy (see Example \ref{Ex: inequality not strict}).
\end{remark}

\bigskip

Now we give an organization of this article: In Section \ref{S: Preliminaries}, we will recall some definitions and known results that will be essential in the proofs of our results.
In Section \ref{S: local rotation type}, we will study the local rotation types for the isotopy of an orientation and area preserving local homeomorphism, and will prove Proposition \ref{P: local rotation type} and Proposition \ref{P: local rotation type non-isolated}.
 In Section \ref{S: property of towsion low isotopy}, we will first study the local rotation set of a torsion-low local isotopy and prove Proposition \ref{P: rotation set of a  torsion-low local isotopy}, then we study the local dynamics of the locally transverse foliation of a torsion-low local isotopy and prove Proposition \ref{P: local dynamic of the transverse foliation of a torsion-low isotopy}.
 In Section \ref{S: existence of a  torsion-low maximal isotopy},  we will prove the existence of a  torsion-low maximal  isotopy (Theorem \ref{T: main}) and will prove Proposition \ref{P: torsion-low isotopy has most fixed points}. In Section \ref{S: examples}, we will give  some explicit examples to show the optimality of our results.
 In Appendix \ref{S: prime ends compactification}, we will give a way to blow up the homeomorphism and get the blow-up rotation number, which will be used in our proof of the main result.
  In Appendix \ref{S: generating fucntion}, we will introduce a way to construct maximal isotopies and  transverse foliations by  generating functions, which will be used in the constructing of our examples.

\section{Preliminaries}\label{S: Preliminaries}

\subsection{Unlinked sets and maximal isotopies}\label{S: pre-Jaulent's preorder}

  In this section, we will recall some results about the isotopies of surface homeomorphisms due to Jaulent \cite{Jaulent} and  B\'eguin, Crovisier and Le Roux \cite{BCLRisotopies}.

  Let $f$ be a homeomorphism of an oriented and connected surface $M$ that is isotopic to the identity.  We say that an identity isotopy $I=(f_t)_{t\in[0,1]}$ of $f$  is an \emph{isotopy  relatively to $X$} if $f_t(x)\equiv x$  for all $x\in X$  and $t\in[0,1]$. We say that $f$ is \emph{isotopic to the identity relatively to $X$} if there exists an identity isotopy  of $f$ relatively to $X$.  We say that a subset $X\subset\mathrm{Fix}(f)$ is \emph{unlinked}  if $f$ is isotopic to the identity relatively to $X$.

   We denote by $(X, I)$ the couple that consists of an unlinked closed subset $X\subset\mathrm{Fix}(f)$ and an identity isotopy $I=(f_t)_{t\in [0,1]}$ of $f$ relatively to $X$. Let $\mathcal{I}$ be the set of such couples with the following preorder: $(X, I)\precsim (Y, I')$, if
\begin{itemize}
\item[i)] $X\subset Y\subset \mathrm{Fix}(f)$ are unlinked,
\item[ii)] for every $z\in M\setminus X$, its  trajectories  along $I'$  and  $I$ are homotopic  in  $M\setminus X$.
\end{itemize}

The preorder $\precsim$ is well defined. Moreover, if  $(X, I)\precsim (Y, I')$ and $(Y, I')\precsim (X, I)$, then one has $X=Y$ and the trajectories of each $z\in M\setminus X$ along $I$ and $I'$ are homotopic in $M\setminus X$. In this case, we  write $(X,I)\sim(Y, I')$, where $\sim$ is an equivalence relation.
We should note the following facts:
\begin{itemize}
  \item[-] if $(X,I)\sim (X, I')$, then for any connected component $M_0$ of $M\setminus X$, the lifts of $f|_{M_0}$ to the universal covering space  obtained by lifting the restricted isotopies respectively coincide.
  \item[-] if $M\setminus X$ is not homeomorphic to an annulus or to a torus, two couples $(X,I)$ and $(X, I')$ are always equivalent (see  \cite[Remark 1.6]{Jaulent} or   \cite[Section 2.3]{BCLRisotopies}).
\end{itemize}

 We call $(Y,I')\in\mathcal{I}$  an \emph{extension} of $(X,I)$ if $(X,I)\precsim (Y, I')$; we call  $I'$  an \emph{extension} of $(X,I)\in\mathcal{I}$ if $(X,I)\precsim(\mathrm{Fix}(I'),I')$; we call $I'$  an \emph{extension} of $I$ if $I'$ is an extension of $(\mathrm{Fix}(I),I)$.  We say that $I'$ is a \emph{maximal  extension} if $(\mathrm{Fix}(I'),I')$ is maximal in $(\mathcal{I},\precsim)$.

We call a fixed point $z$ of $f$ a \emph{contractible} fixed point associated to $I$, if the trajectory of $z$ along $I$ is homotopic to zero in $M$.  By  \cite[Lemma A.8]{BCLRisotopies}, we have the following result:

\begin{proposition}\label{P: pre-maximal implies no contractible fixed point}
Let $(X,I)\in\mathcal{I}$. If $f|_{M\setminus X}$ has a contractible fixed point $z$ associated to $I|_{M\setminus X}$, then there exists $(X\cup\{z\}, I')\in\mathcal{I}$ such that  $(X, I)\precsim(X\cup\{z\}, I')$. In particular, if $(X,I)$ is  maximal in $(\mathcal{I},\precsim)$, $f|_{M\setminus X}$ has no contractible fixed point associated to $I_{M\setminus X}$.
\end{proposition}

We have the following  result from  \cite[Theorem 1 and Corollary 2.9]{BCLRisotopies}.

\begin{proposition}\label{P: pre-exist uper bound of  isotopies sequence}
If $\{(X_\alpha, I_{\alpha})\}_{\alpha\in J}$ is a totally ordered chain  in  $(\mathcal{I},\precsim)$, then there exists  $(X_{\infty},I_{{\infty}})\in\mathcal{I}$ that is an upper bound of the this chain, where $X_{\infty}=\overline{\cup_{\alpha\in J}X_\alpha}$
\end{proposition}

Also, we need  the following two results.

\begin{proposition}\label{T: pre-exist max isotopy}(\cite[Corollary 1.3]{BCLRisotopies})
For every $(X,I)\in\mathcal{I}$, there exists a maximal element $(X', I')\in\mathcal {I}$ such that $(X,I)\precsim (X', I')$.
\end{proposition}

\begin{proposition}\label{P: pre-connected implies unlinked in plane}(\cite[Corollary 1.1]{BCLRisotopies})
If $f$ is an orientation preserving homeomorphism of the plane, and if $X\subset \mathrm{Fix}(f)$ is a connected and closed subset, then there exists an identity isotopy $I$ of $f$ such that $X\subset \mathrm{Fix}(I)$.
\end{proposition}

\subsection{Local dynamics of transverse foliations}\label{S: pre-local dynamics of transverse foliation}

Let $f:(W,z_0)\rightarrow (W',z_0)$ be an orientation preserving local homeomorphism at $z_0$, $I$ a local isotopy of $f$, and  $\mathcal{F}$ a locally transverse foliation of $I$.
In this section, we will describe the local properties of $\mathcal{F}$ near  $z_0$.

 We call $z_0$ a \emph{sink}  (resp. a\emph{ source}) if  there is a homeomorphism $h: U\rightarrow \mathbb{D}$ which sends  $z_0$ to $0$ and sends the restricted foliation $\mathcal{F}|_{U\setminus\{z_0\}}$ to the radial foliation of $\mathbb{D}\setminus\{0\}$ with the leaves toward (resp. backward) $0$, where $U$ is a neighborhood of $z_0$ and $\mathbb{D}$ is the unit disk. A \emph{petal} of $\mathcal{F}$ is a closed topological disk whose boundary is the union of a leaf and a singularity. Let $\mathcal{F}_0$ be the foliation on $\mathbb{R}^2\setminus\{0\}$ whose leaves are the horizontal lines except the $x-$axis which is cut  into two leaves. Let $S_0=\{y\ge 0:x^2+y^2\le 1\}$ be the half-disk.  We call a closed topological disk $S$ a \emph{hyperbolic sector} if there exist
\begin{itemize}
\item[-] a closed set $K\subset S$ such that $K\cap\partial S$ is reduced to a singularity $z_0$ and $K\setminus\{z_0\}$ is the union of the leaves of $\mathcal{F}$ that are contained in $S$,
\item[-] a continuous map $\phi: S\rightarrow S_0$ that maps $K$ to $0$ and the leaves of $\mathcal{F}|_{S\setminus K}$ to the leaves of $\mathcal{F}_0|_{S_0}$.
\end{itemize}
\begin{figure}[h]
  \subfigure[the hyperbolic sector model $S_0$]{
    \begin{minipage}[t]{0.3\linewidth}
      \centering
   \includegraphics[width=3cm]{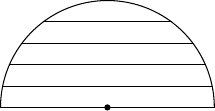}
    \end{minipage}}
  \subfigure[a pure hyperbolic sector]{
    \begin{minipage}[t]{0.3\linewidth}
      \centering
     \includegraphics[width=2cm]{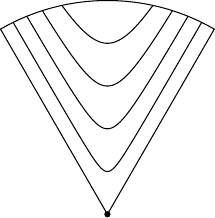}
    \end{minipage}}
    \subfigure[a strange hyperbolic sector]{
    \begin{minipage}[t]{0.3\linewidth}
      \centering
   \includegraphics[width=2cm]{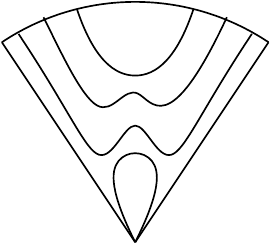}
    \end{minipage}}
  \caption{The hyperbolic sectors}
\end{figure}
The index $i(\mathcal{F},z_0)$ of $\mathcal{F}$  is  well defined (see p.150 of \cite{lecalveztourner} for the precise definition). Due to Le Roux,  we have the following proposition (see  \cite[ Appendix B]{lerouxrotation}):
\begin{proposition}\label{P: pre-dynamics of a foliation}
 We have one of the following cases:
 \begin{itemize}
 \item[i)](sink or source) there exists a neighborhood of $z_0$ that contains neither a closed leaf, nor a petal, nor a hyperbolic sector;
 \item[ii)](cycle) every neighborhood of $z_0$ contains a closed leaf;
 \item[iii)](petal) every small neighborhood of $z_0$ contains a petal, and does not contain any hyperbolic sector;
  \item[iv)](saddle) every small neighborhood of $z_0$ contains a hyperbolic sector,  and does not contain any petal;
 \item[v)](mixed) every neighborhood of $z_0$ contains both a petal and a hyperbolic sector.
 \end{itemize}
Moreover, $i(\mathcal{F},z_0)$ is equal to $1$ in the first two cases,  is strictly larger than $1$ in the petal case, and is strictly smaller than $1$ in the saddle case.
\end{proposition}

\begin{remark}\label{R: pre-local dynamics of foliation-area preserving}
In particular, when $f$ is area preserving, $z_0$ is a sink, or a source, or a saddle of $\mathcal{F}$.
\end{remark}

\begin{remark}
 Note that for a maximal isotopy $I$ and a transverse foliation $\mathcal{F}$ with an isolated singularity $z_0$, if we view  $I$ as a local isotopy at $z_0$, then $\mathcal{F}$ is also locally transverse to $I$  except in the case where $\mathcal{F}$ does not have any other singularity and  $M\setminus\{z_0\}$ is homeomorphic to $\mathbb{R}^2$ (\cite[Proposition 3.4]{lecalveztourner}). In particular, if the homeomorphism is area preserving, then $\mathcal{F}$ is locally transverse to $I$ at every isolated singularity, and we can give a classification of the isolated singularities of $\mathcal{F}$.
\end{remark}

Furthermore, when $z_0$ is an isolated fixed point of $f$, we have the following proposition. The first part of case i) is a direct corollary \cite[Proposition 1.1 and Proposition 3.3]{lecalveztourner} (or see  \cite[Theorem 4.1]{lerouxrotation}), the second part of case i)  is just Proposition 3.7 of \cite{lecalveztourner}, and the case ii) is a direct corollary of  \cite[ Remark 1.2  and Proposition 3.3]{lecalveztourner}.
\begin{proposition}\label{P: pre-index and rotation type}
Let $f:(W,z_0)\rightarrow (W',z_0)$ be an orientation preserving local homeomorphism at  an isolated fixed point $z_0$.

 \begin{itemize}
\item[i)]
 When the Lefschez index $i(f,z_0)\ne 1$, there exists a  local isotopy $I_0$ of $f$ such that for all locally transverse foliations $\mathcal{F}_0$ of $I_0$, $i(\mathcal{F}_0,z_0)=i(f,z_0)\ne 1$.

Moreover, for any local isotopy $I$  of $f$ and any locally transverse  foliation $\mathcal{F}$ of $I$, we have
\begin{itemize}
 \item[-] $i(\mathcal{F},z_0)=i(\mathcal{F}_0,z_0)$, if $I\sim I_0$;
 \item[-] $z_0$ is a sink of $\mathcal{F}$, if $I>I_0$;
 \item[-] $z_0$ is a source of $\mathcal{F}$, if $I<I_0$.
 \end{itemize}
 \item[ii)] When the Lefschez index $i(f,z_0)=1$, for all local isotopy $I$ of $f$ and all locally transverse foliation $\mathcal{F}$ of $I$, we have
     $i(\mathcal{F},z_0)=i(f,z_0)=1$.
 \end{itemize}
\end{proposition}

\subsection{The blow-up at a fixed point}\label{S: pre-blow up}

Let $f:(W,z_0)\rightarrow (W',z_0)$ be an orientation preserving local homeomorphism.
We say that $f$ can be \emph{blown  up} at $z_0$, if we can replace $z_0$ by a unit circle $S^1$ and extend $f|_{W\setminus\{z_0\}}$ continuously to a homeomorphism between $W\setminus \{z_0\}\sqcup S^1$ and $W'\setminus\{z_0\}\sqcup S^1$. In particular, when $f$ is a diffeomorphism,  the extension  can be induced by the map
\[v\mapsto \frac{Df(z_0) v}{\|Df(z_0)v\|}\]
on the space of unit tangent vectors.

Suppose that $f$ can be blown up at $z_0$, and that $f$ is not conjugate to a contraction  or an expansion. Let us denote by $h$ the extension of $f$ on $S^1$.  Let $I=(f_t)_{t\in[0,1]}$ be a local isotopy of $f$. We choose a small disk $D$ containing $z_0$ and consider the universal cover  $\pi: \widetilde{D} \rightarrow D\setminus\{z_0\}$. Let $(\widetilde{f}_t)_{t\in[0,1]}$ be the lift of $(f_t)_{t\in[0,1]}$  such that $\widetilde{f}_0$ is equal to the identity in its domain of definition.  Let $\widetilde{h}$ be the lift of   $h$ to $\mathbb{R}$ that is a continuous extension of $\widetilde{f}_1$. We define the \emph{blow-up rotation number} $\rho(I,z_0)\in \mathbb{R}$ to  the rotation number of $h$ associated to the lift $\widetilde{h}$.  Jean-Marc Gambaudo, Le Calvez, and Elisabeth P\'ecou \cite{LecalvezthmNaishul} proved that the blow-up rotation numbers do not depend on the choice of $h$,   which generalizes a previous result of Na\u{\i}shul$'$ \cite{Naishul}.

Le Roux \cite{lerouxrotation} studied several cases where $f$ can be blown up. We will need the following one in this article:

\begin{proposition}\label{P: pre-blow-up}
If there exists a locally transverse foliation $\mathcal{F}$ of $I$ that has a leaf $\gamma^+$ from $z_0$ and a leaf $\gamma^-$ toward $z_0$, then $f$ can be blown up at $z_0$. Moreover, under the further assumption that $f$ is not conjugate to a contraction  or an expansion,  the blow-up rotation number $\rho(I,z_0)$ is equal to $0$.
\end{proposition}

\begin{remark}
In particular, we are in this case if $z_0$ is a petal, a saddle or a mixed singularity of $\mathcal{F}$.
\end{remark}

\subsection{The local rotation set}\label{S: pre-local rotation set}

In this section, we recall the definition of the local rotation set by Le Roux  and   describe some of its properties. More details can be found in \cite{lerouxrotation}.

Let $f: (W,0)\rightarrow (W',0)$ be an orientation preserving local homeomorphism at $0\in\mathbb{R}^2$, and $I=(f_t)_{t\in[0,1]}$ be a local isotopy of $f$. Given two neighborhoods $V\subset U$ of $0$ and an integer $n\ge 1$, we define
\[E(U,V,n)=\{z\in U: z\notin V,f^{n}(z)\notin V, f^i(z)\in U \textrm{ for all } 1\le i\le n\}.\]
We define the \emph{rotation set} of $I$ relative to $U$ and $V$ by
\[\rho_{U,V}(I)=\cap_{m\ge 1}\overline{\cup_{n\ge m}\{\rho_n(z): z\in E(U,V,n)\}}\subset[-\infty,\infty],\]
where $\rho_n(z)$ is the average change of  angular coordinate along the trajectory of $z$ during $n$ iterates.
We define the \emph{local rotation set} of $I$ to be
\[\rho_s(I,0)=\cap_U\overline{\cup_V\rho_{U,V}(I)}\subset[-\infty,\infty],\]
where $V\subset U\subset W$ are neighborhoods of $0$.

More generally, when $f:(W,z_0)\rightarrow (W',z_0)$ is an orientation preserving local homeomorphism at $z_0$, we can define the local rotation set  by conjugating $f$ to an orientation preserving local homeomorphism at $0\in\mathbb{R}^2$.

The local rotation set can be empty. However, Le Roux (see \cite{lerouxparabolic} or \cite[Theorem 2.4]{lerouxrotation}) proved that the local rotation set is not empty  unless  the local homeomorphism is conjugate to one of the  following maps:
\begin{itemize}
\item[-] the contraction $z\mapsto\frac{z}{2}$,
\item[-] the expansion $z\mapsto 2z$,
\item[-] a holomorphic function  $z\mapsto e^{i2\pi\frac{p}{q}}z(1+z^{qr})$ with $q,r\in\mathbb{N}^+$ and $p\in\mathbb{Z}$.
\end{itemize}
In particular, when $f$ is area preserving, none of the obove three cases occurs and  the local rotation set is not empty.

\bigskip

We say that $z$ is a \emph{contractible} fixed point of $f$ associated to the local isotopy $I=(f_t)_{t\in[0,1]}$ if the trajectory $t\mapsto f_t(z)$  of $z$ along $I$ is a loop homotopic to zero in $W\setminus\{z_0\}$. We say that $f$ satisfies  the \emph{local intersection condition}, if there exists a neighborhood $V$ of $z_0$, such that  for any Jordan domain $D\subset V$ containing $z_0$, $f(\partial D)\cap \partial D\ne \emptyset$. Note that if $f$ satisfies the local intersection condition, then so does $f^n$ (see  \cite[Lemma 3.5]{lerouxrotation} for example).
In particular, if $f$ is area preserving, then $f$ satisfies the local intersection condition. For  two isotopies (resp. local isotopies) $I=(f_t)_{t\in [0,1]}$ and $I'=(g_t)_{t\in[0,1]}$, we denote by $I^{-1}$ the  isotopy (resp. local isotopy) $(f^{-1}_t)_{t\in [0,1]}$, by $I'I$ the  isotopy (resp. local isotopy) $(\varphi_t)_{t\in[0,1]}$ such that
\[\varphi_t=\left\{\begin{array}{ll} f_{2t} &\mbox{for $t\in[0,\frac{1}{2}]$},\\ g_{2t-1}\circ f &\mbox{for $t\in[\frac{1}{2},1]$},\end{array}\right.\]
and by $I^n$ the isotopy (resp. local isotopy) $\underbrace{I\cdots I}_{n \textrm{ times}}$ for every $n\ge 1$. The local rotation set satisfies the following properties:

\begin{proposition}(\cite{lerouxrotation}) \label{P: pre-rotation set}
Let $f:(W,z_0)\rightarrow (W',z_0)$ be an orientation preserving local homeomorphism, and $I$  a local isotopy of $f$. One has the following results:
\begin{itemize}
\item[i)] for all integer $p,q$, $\rho_s(J_{z_0}^pI^q,z_0)=q\rho_s(I,z_0)+p$, where $J_{z_0}=(R_{2\pi t})_{t\in[0,1]}$ is the local isotopy of the identity and each $R_{2\pi t}$ is the counter-clockwise rotation through an angle $2\pi t$ about the center $z_0$;
\item[ii)] if $z_0$ is accumulated by contractible fixed points of $f$ associated to $I$, then $0\in \rho_s(I,z_0)$;
\item[iii)]if $f$ satisfies the local intersection condition and if $0$ is an interior point of the convex hull of $\rho_s(I,z_0)$, then $z_0$ is accumulated by contractible fixed points of $f$ associated to $I$;
\item[iv)] if $I$ has a positive (resp. negative) rotation type, then $\rho_s(I,z_0)\subset[0, +\infty]$ (resp. $\rho_s(I,z_0)\subset [-\infty, 0]$);
\item[v)] if $\rho_s(I,z_0)$ is a non-empty set that is contained in $(0,+\infty]$ (resp. $[-\infty, 0)$), then $I$ has a positive (resp.negative) rotation type;
\item[vi)] if $\rho_s(I,z_0)$ is  a non-empty set that is contained in $[0,\infty]$ (resp. $[-\infty, 0]$)  and is not reduced to  $0$, and  if $z_0$ is not accumulated by contractible fixed points of $f$ associated to $I$, then $I$ has a  positive (resp. negative) rotation type;
\item[vii)] if $f$ can be blown up at $z_0$, and if $\rho_s(I,z_0)$  is not empty, then  $\rho_s(I,z_0)$ is  reduced to the blow-up rotation number $\rho(I,z_0)$.
\end{itemize}
\end{proposition}

\begin{remark}\label{R: pre-rotation set like interval}
When $f$ satisfies the local intersection condition, one can deduce that $\rho_s(I,z_0)$ is a closed interval ($\subset [-\infty,+\infty]$) as a corollary of the assertions i), ii), iii) of the  proposition; if we add the assumption that $z_0$ is an isolated fixed point, then  $\rho_s(I,z_0)$ is a closed subinterval of $ [k,k+1]$ for an integer $k$ by the assertions i) and iii) of the proposition.
\end{remark}

\section{The local rotation type}\label{S: local rotation type}
\subsection{The rotation type at an isolated fixed point}\label{S: rotation type at an isolated fixed point}

Let $f: (W,z_0)\rightarrow (W',z_0)$ be an orientation and area  preserving local homeomorphism with an isolated fixed point $z_0$. In this section, we will detect the local rotation type of local isotopies of $f$ and prove Proposition \ref{P: local rotation type}.

\begin{proof}[Proof of Proposition \ref{P: local rotation type}]
Let $I$ be a local isotopy of $f$, and $\mathcal{F}$ a locally transverse foliation of $I$.
We will prove the proposition in two cases.
 \begin{itemize}
 \item[i)]
Suppose that $i(f,z_0)\ne 1$.

By the first part of case i) of Proposition \ref{P: pre-index and rotation type}, there exists a  local isotopies $I_0$ of $f$ such that for every locally transverse foliation $\mathcal{F}_0$ of $I_0$, $i(\mathcal{F}_0,z_0)=i(f,z_0)\ne 1$, and hence $z_0$ is a saddle of $\mathcal{F}_0$ (see Proposition \ref{P: pre-dynamics of a foliation} and Remark \ref{R: pre-local dynamics of foliation-area preserving}). By the second part of case i) of  Proposition \ref{P: pre-index and rotation type},
\begin{itemize}
 \item[-] if $I\sim I_0$, then  $i(\mathcal{F},z_0)=i(\mathcal{F}_0,z_0)\ne 1$, and hence $z_0$ is a saddle of $\mathcal{F}$;
 \item[-] if $I>I_0$, then $z_0$ is a sink of $\mathcal{F}$;
 \item[-] if $I<I_0$, then $z_0$ is a source of $\mathcal{F}$.
 \end{itemize}
 Therefore,  for a local isotopy $I$  at $0$, it has  a zero rotation type if it is in the homotopy class of $I_0$; it has  a positive rotation type if $I>I_0$; and it has  a negative rotation type if $I<I_0$. Both statements of Proposition \ref{P: local rotation type} are proven.

\item[ii)]
Suppose that $i(f,z_0)=1$.

By the case ii) of  Proposition \ref{P: pre-index and rotation type}, we have  $i(\mathcal{F},z_0)=1$. By Proposition \ref{P: pre-dynamics of a foliation} and Remark \ref{R: pre-local dynamics of foliation-area preserving}, $z_0$ is a sink or a source of $\mathcal{F}$. For any locally transverse  foliation $\mathcal{F'}$ of $I$, Matsumoto \cite{Matsumoto} proved in fact that $z_0$ is a sink (resp. source) of $\mathcal{F}'$ if it is a sink (resp. source) of $\mathcal{F}$.  So, a local isotopy $I$ of $f$ has either a positive or a negative rotation type, and has exactly one of the two rotation types.  We have proven the first statement of Proposition \ref{P: local rotation type}.

 Since $f$ is area preserving, $\rho_s(I,0)$ is not empty.  Moreover, since $z_0$ is an isolated fixed point, we know by Remark \ref{R: pre-rotation set like interval} that there exists $k\in \mathbb{Z}$ such that $\rho_s(I,0)\subset [k,k+1]$. By the assertion i) of Proposition \ref{P: pre-rotation set},  there exists a local isotopy $I_0$ of $f$  such that $\rho_s(I_0,0)$ is a nonempty subset of $[0,1]$ and is not reduced to $1$. Then, as a corollary of the assertions i), v) and vi) of Proposition \ref{P: pre-rotation set}, we know that
\begin{itemize}
\item[-]$I$ has a positive rotation type if $I> I_0$,
\item[-]$I$ has a negative rotation type if $I< I_0$.
\end{itemize}
Therefore, $I_0$ is torsion-low at $z_0$.
Recall that  $I_0$ has either a positive or a negative rotation type. We have proven the second statement of Proposition \ref{P: local rotation type}.\qedhere
\end{itemize}
\end{proof}

\subsection{The local rotation type at a non-isolated fixed point}

Let $f: (W,z_0)\rightarrow (W',z_0)$ be an orientation and area  preserving local homeomorphism with a non-isolated fixed point $z_0$. We will prove Proposition \ref{P: local rotation type non-isolated} as a corollary of Proposition \ref{P: pre-rotation set}.

\begin{proof}[Proof of Proposition \ref{P: local rotation type non-isolated}]
If $I$ has a positive  rotation type, there exists a foliation that is locally transverse to $I$ such that $z_0$ is a sink of $\mathcal{F}$.  For any  fixed  point $z$ of $f$ in a sufficient small neighborhood of $z_0$, the trajectory of $z$ along $I$ is homotopic to a path that is locally transverse  to $\mathcal{F}$, and hence is not contractible in $W\setminus\{z_0\}$. So, $z_0$ is not accumulated by contractible fixed points of $f$ associated to $I$. Moreover, since $z_0$ is accumulated by fixed points of $f$, there exists a sequence $\{z_n\}_{n}$ of fixed points of $f$ that tends to $z_0$. For each $n$, the trajectory of $z_n$ along $I$ is homotopic to a path that is locally transverse  to $\mathcal{F}$, and hence the linking number $L(I,z_n,z_0)$,   that is the index of the trajectory of $z_n$ along $I$ relatively to $z_0$, is positive. By the definition of the local rotation set,    $\rho_s(I,z_0)$  contains the limit points of $\{L(I,z_n, z_0): n\in\mathbb{N}\}$ as a subset of $[-\infty,+\infty]$, so $\rho_s(I,z_0)$ contains a positive integer or $+\infty$. By iv) of Proposition \ref{P: pre-rotation set}, $\rho_s(I,z_0)\subset [0,+\infty]$. So, $\rho_s(I,z_0)$ is a subset of $[0,+\infty]$ that contains a positive integer or $+\infty$.

Similarly, if $I$ has a negative  rotation type, then $z_0$ is not accumulated by contractible fixed points of $f$ associated to $I$, and  $\rho_s(I,z_0)$ is a subset of $[-\infty, 0]$ that contains a negative integer or $-\infty$. So, $I$ can not have both a positive and a negative rotation type.

We will finish the proof of the first assert of the proposition by the following argument. If $z_0$ is not accumulated by contractible fixed points of $f$ associated to $I$, there is a foliation $\mathcal{F}$ that is locally transverse to $I$. By Remark \ref{R: pre-local dynamics of foliation-area preserving}, $z_0$ is a sink, or a source, or a saddle of $\mathcal {F}$. The case $z_0$ is a saddle of $\mathcal{F}$ cannot occur, because in this case, there is no fixed point of $f$ in sufficiently small neighborhood of $z_0$. Now, we know that $z_0$ is a sink, or a source of $\mathcal {F}$, and hence $I$ has a positive or a negative rotation type.

If $I$ has a positive rotation type, $\rho_s(I,z_0)$ is a subset of $[0,+\infty]$ that contains a positive integer or $+\infty$. For any local isotopy $I'>I$, by i) of Proposition \ref{P: pre-rotation set},  there exists a positive integer $k$ such that $\rho_s(I',z_0)\subset [k,+\infty]$ is a non-empty set. By v) of Proposition \ref{P: pre-rotation set}, $I'$ has a positive rotation type.

Similarly, we can prove  if $I$ has a negative rotation type, every local isotopy $I'<I$
has a negative rotation type.

If  $z_0$ is accumulated by contractible fixed points of $f$ associated to $I$, by assertion ii) of Proposition \ref{P: pre-rotation set}, $0\in \rho_s(I,z_0)$. For any local isotopy $I'>I$, by i) of  Proposition \ref{P: pre-rotation set},  $\rho_s(I',z_0)$ contains a positive integer. So, either $z_0$ is accumulated by contractible fixed points of $f$ associated to $I'$, or by iii) and vi) of  Proposition \ref{P: pre-rotation set}, $I'$ has a positive rotation type. Similarly, for any local isotopy $I'<I$, either situation i) or iii) of Proposition \ref{P: local rotation type non-isolated} occurs.
\end{proof}

\section{The properties of a  torsion-low local isotopy}\label{S: property of towsion low isotopy}
\subsection{The local rotation set of a torsion-low local isotopy }

Let $f: (W,z_0)\rightarrow (W',z_0)$ be an orientation and area  preserving local homeomorphism and $I$ a torsion-low local isotopy of $f$.
In this section, we will study the local rotation set of a torsion-low isotopy and prove Proposition \ref{P: rotation set of a  torsion-low local isotopy}.
\begin{proof}[Proof of Proposition \ref{P: rotation set of a  torsion-low local isotopy}]
The first two statements are just a corollary of the definition of the torsion-low property and the assertions i), ii), iv) of Proposition \ref{P: pre-rotation set}.

Suppose now that $f$ can be blown up at $z_0$. Since $f$ is area preserving, $\rho_s(I,z_0)$ is not empty.  So, using the assertion vii) of Proposition \ref{P: pre-rotation set}, one deduces that $\rho_s(I,z_0)$ is reduced to a single point in $[-1,1]$.

Suppose now that $f$ is a diffeomorphism in a neighborhood of $z_0$. The first part of the fourth statement is just a special case of the third statement.

 To conclude, let us prove the last part of the fourth statement. To simplify the notations, we suppose that $z_0=0\in\mathbb{R}^2$. Since $f$ is area preserving,
  $Df(0)$ can not have two real eigenvalues such that the absolute values of both eigenvalues are strictly smaller (resp. larger) than $1$.  Since $1$ is not an eigenvalue of $Df(0)$, one has to consider the following three cases:
\begin{itemize}
\item[-] $Df(0)$ do not have any real eigenvalue. In this case, $\rho(I,0)$ is not an integer.
\item[-]  $Df(0)$ has two real eigenvalues $\lambda_1$ and $\lambda_2$ such that $\lambda_1<-1<\lambda_2<0$. In this case, $\rho(I,0)$ is equal to $\frac{1}{2}$ or $-\frac{1}{2}$.
\item[-]  $Df(0)$ has two real eigenvalues $\lambda_1$ and $\lambda_2$ such that $0<\lambda_1<1<\lambda_2$. In this case, $i(f,0)=-1$, and $I$ has a zero rotation type at $0$. By Proposition \ref{P: pre-blow-up},   $\rho(I,0)$ is equal to $0$.
\end{itemize}
In any case, we obtain  that $\rho(I,0)$ belongs to $(-1,1)$.
\end{proof}

\subsection{Local dynamics  of the locally transverse foliation of a torsion-low local isotopy}\label{S: Local dynamics  of the locally transverse foliation of a torsion-low local isotopy}

Let $f: (W,z_0)\rightarrow (W',z_0)$ be an orientation and area  preserving local homeomorphism, $I$ a torsion-low local isotopy, and $\mathcal{F}$ a locally transverse foliation of $I$. Here, $z_0$ could be an isolated or a non-isolated fixed point of $f$, but it is an isolated singularity of $\mathcal{F}$ by definition of locally transverse foliation.  In this section, we will give a discription of the local dynamics of $\mathcal{F}$. Formally, we will prove the following result, which immediately implies Proposition \ref{P: local dynamic of the transverse foliation of a torsion-low isotopy}.

\begin{proposition}\label{P: local dynamics of transverse foliation-torsion low 2}
Under the previous assumptions,
\begin{itemize}
\item[-] If $z_0$ is an isolated fixed point of $f$ such that $i(f,z_0)\ne 1$, then $z_0$ is a saddle of $\mathcal{F}$ and $i(\mathcal{F},z_0)=i(f,z_0)$.
\item[-]  If $z_0$ is an isolated fixed point such that $i(f,z_0)=1$ or if $z_0$ is not isolated in $\mathrm{Fix}(f)$, then $z_0$ is a sink or a source of $\mathcal{F}$.
\end{itemize}
 \end{proposition}

\begin{proof}
One has to consider two cases: $z_0$ is an isolated fixed point of $f$ or not.
\begin{itemize}
\item[i)] Suppose that $z_0$ is an isolated fixed point of $f$. As a corollary of Proposition \ref{P: local rotation type},
\begin{itemize}
\item[-] $z_0$ is a saddle of $\mathcal{F}$ if $i(f,z_0)\ne 1$;
\item[-] $z_0$ is a sink or a source of $\mathcal{F}$ if $i(f,z_0)=1$.
\end{itemize}
Moreover, in the first case, as in the first part of the proof of Proposition \ref{P: local rotation type}, we have $i(\mathcal{F},z_0)=i(f,z_0)$.

\item[ii)] Suppose that $z_0$ is a non-isolated fixed point of $f$.  By Proposition \ref{P: local rotation type non-isolated}, $I$ has a positive or negative rotation type, so $z_0$ is a sink or a source of $\mathcal{F}$.
\qedhere
\end{itemize}
\end{proof}

\section{The existence of a torsion-low maximal isotopy}\label{S: existence of a  torsion-low maximal isotopy}

Let $f$ be an area preserving homeomorphism of  a connected oriented surface $M$ that is isotopic to the identity.  The main aim of this section is to prove the existence of a torsion-low maximal isotopy of $f$, i.e. Theorem \ref{T: main}.

 When $\mathrm{Fix}(f)=\emptyset$, the theorem is trivial. So we suppose that $\mathrm{Fix}(f)\ne \emptyset$ in the following part of this section. Recall that $\mathcal{I}$ is the set of couples $(X, I)$ that consists of a closed unlinked subset $X\subset \mathrm{Fix}(f)$ and  an  isotopy $I$  from the identity to  $f$  relatively to $X$. We denote by $\mathcal{I}_0$  the set of $(X,I)\in\mathcal{I}$  such that $I$ is torsion-low at every $z\in X$.  Recall that $\precsim$ is the preorder defined in Section \ref{S: pre-Jaulent's preorder}. Then, Theorem \ref{T: main} is just an immediate corollary of the following theorem.

\begin{theorem}\label{T: existence of globle torsion-low isotopy}
Given $(X,I)\in\mathcal{I}_0$, there exists a maximal extension $(X', I')$ of $(X,I)$ that belongs to $\mathcal{I}_0$.
\end{theorem}

\begin{remark}\label{R: globle torsion-low 1}
 We will see that for every $z\in X$, $I'$ and $I$ are locally  homotopic   as  local isotopies at $z$, except in the case that $M$ is a sphere and $X$ is reduced to a point. In the case  that  $M$ is a sphere and $X$ is reduced to one point, this is not necessary the case (see Example \ref{Ex: local torsion low not imply globle}).
\end{remark}

\begin{remark}\label{R: globle torsion-low 2}
One may fail to find a torsion-low maximal  isotopy $I$ such that  every $z\in \mathrm{Fix}(I)$ that is not isolated in $\mathrm{Fix}(f)$ is also not isolated in $\mathrm{Fix}(I)$. We will give an example (Example \ref{Ex: zero not always in rotation set}) in Section \ref{S: examples}. In particular, in this example, for every torsion-low maximal  isotopy, there is a point that is isolated in $\mathrm{Fix}(I)$ but is not isolated in $\mathrm{Fix}(f)$.
\end{remark}

\begin{remark}\label{R: not torsion low means isolated}
We also note the following fact which results immediately from the definition of torsion-low property and Proposition \ref{P: local rotation type non-isolated}:\\
If $(Y,I)\in\mathcal{I}$ and $z\in Y$ is a point such that $I$ is not torsion-low at $z$,  then $z$ is isolated in $Y$.
\end{remark}

Now, we begin the proof of  Theorem \ref{T: existence of globle torsion-low isotopy}. It  is a corollary of Zorn's lemma and  the following  Propositions \ref{P: mainproof-existence of an upper bound}-\ref{P: mainproof-sphere}.  We will first deduce the theorem from the four propositions, then we will prove the  propositions one by one.

\begin{proposition}\label{P: mainproof-existence of an upper bound}
If $\{(X_{\alpha}, I_{\alpha})\}_{\alpha\in J}$ is a totally ordered chain in $\mathcal{I}_0$, then there exists an upper bound  $(X_{\infty},I_{\infty})\in\mathcal{I}_0$ of this chain, where $X_{\infty}=\overline{\cup_{\alpha\in J}X_{\alpha}}$
\end{proposition}

\begin{proposition}\label{P: mainproof-generalcase}
Consider a maximal $(Y, I)\in\mathcal{I}$ and $z\in Y$  such that $I$ is not torsion-low at $z$.
When the underlying surface $M$ is a sphere, we suppose that $Y\setminus \{z\}$ contains at least two points; and when $M$ is a plane, we suppose that $Y\setminus\{z\}$ is not empty. Then, there always exists a maximal extension $(Y', I')\in\mathcal{I}$ of $(Y\setminus\{z\},I)$ and $z'\in Y'\setminus (Y\setminus\{z\})$ such that $I'$ is torsion-low at $z'$.
\end{proposition}

\begin{proposition}\label{P: mainproof-plane}
When $M$ is a plane and $f$ admits a fixed point,  $(X,I)\in\mathcal{I}_0$ is not maximal in $(\mathcal{I}_0,\precsim)$ if $X=\emptyset$.
\end{proposition}

\begin{proposition}\label{P: mainproof-sphere}
When $M$ is a sphere, $(X,I)\in\mathcal{I}_0$ is not maximal in $(\mathcal{I}_0,\precsim)$ if $\#X\le 1$.
\end{proposition}

\begin{remark}
Proposition \ref{P: mainproof-plane} and \ref{P: mainproof-sphere} deal with two special cases. The first one is easy, while the second one is more difficult. Indeed, to find  an identity isotopy on a plane that is torsion-low at one point, we do not need to study  the dynamics at infinity; but to find  an identity isotopy on a sphere that is torsion-low at two points, we need to check the properties of the isotopy near both points.
\end{remark}

 \begin{proof}[Proof of Theorem \ref{T: existence of globle torsion-low isotopy}]

   Fix $(X,I)\in\mathcal{I}_0$. Let $\mathcal{I}_*$ be the set of equivalent classes of the extensions  $(X',I')\in\mathcal{I}_0$ of $(X,I)$. Then, the preorder $\precsim$ induces a partial order over $\mathcal{I}_*$. To simplify the notations, we still denote  by $\precsim$ this partial order.  By Proposition \ref{P: mainproof-existence of an upper bound}, $(\mathcal{I}_*,\precsim)$ is a partial ordered set satisfying the condition of Zorn's lemma, so  $(\mathcal{I}_*,\precsim)$ contains at least one maximal element. Choose one representative $(X',I')$ of a maximal element of $(\mathcal{I}_*,\precsim)$. It is an extension of $(X,I)$ and is maximal in $(\mathcal{I}_0,\precsim)$.

  Using Propositions \ref{P: mainproof-generalcase}-\ref{P: mainproof-sphere}, we will prove by contradiction that a maximal couple $(X',I')\in(\mathcal{I}_0,\precsim)$  is also maximal in $(\mathcal{I},\precsim)$. Suppose that there exists a couple $(X',I')\in \mathcal{I}_0$ that is maximal in $(\mathcal{I}_0,\precsim)$ but is not maximal in $(\mathcal{I},\precsim)$.  Fix a maximal extension $(Y,I'')$  of $(X',I')$ in $(\mathcal{I},\precsim)$, and $z\in Y\setminus X'$. Then,  $I''$ is not torsion-low at $z$, and so $z$ is isolated in $Y$ (see Remark \ref{R: not torsion low means isolated}).  Write $Y_0=Y\setminus\{z\}$.  When $M$ is a sphere, by Proposition \ref{P: mainproof-sphere}, $Y_0\supset X'$ contains at least $2$ points; when $M$ is a plane, by Proposition \ref{P: mainproof-plane}, $Y_0\supset X'$ is not empty. By Proposition \ref{P: mainproof-generalcase}, there exist a maximal extension $(Y',I''')$ of $(Y_0,I'')$ and $z'\in Y'$, such that $I'''$ is torsion-low at $z'$.  Note that $(Y', I''')$ is also an extension of $(X',I')$.
  Moreover, for every $z''\in X'$, $I'$ and $I'''$ are equivalent as local isotopies at $z''$. So, $(X'\cup\{z'\},I''')\in\mathcal{I}_0$ is an extension of $(X',I')$, which contradicts with the maximality of $(X',I')$ in $(\mathcal{I}_0,\precsim)$.
\end{proof}

\begin{proof}[Proof of Proposition \ref{P: mainproof-existence of an upper bound}]
By Proposition \ref{P: pre-exist uper bound of  isotopies sequence}, we know that there exists an upper bound $(X_{\infty},I_{\infty})\in\mathcal{I}$ of the chain, where $X_{\infty}=\overline{\cup_{\alpha\in J}X_{\alpha}}$. We only need to prove that  $(X_{\infty},I_{\infty})\in\mathcal{I}_0$.

Without loss of generality, we suppose that the chain  $\{(X_{\alpha}, I_{\alpha})\}_{\alpha\in J}$ is strictly increasing.   When $J$ is finite, the result is obvious. We suppose that $J$ is infinite.  Consider a point $z\in X_{\infty}$. Either $z$ is a limit point of $X_{\infty}$, or  there exists $\alpha_0\in J$ such that  $z$ is an isolated point of $ X_{\alpha}$   for  all $\alpha\in J$ satisfying $(X_{\alpha_0}, I_{\alpha_0})\precsim(X_\alpha, I_{\alpha})$. In the first case, $I_{\infty}$ is  torsion-low at $z$ (Remark \ref{R: ac means torsion-low}); in the second case, by choosing suitable $(X_{\alpha},I_{\alpha})$, we can suppose that $\# X_{\alpha}\ge 2$, then  $I_{\infty}$ and $I_{\alpha}$ are equivalent as local isotopies at $z$. In both case, $I_{\infty}$ is  torsion-low at $z$.
\end{proof}

Before proving Proposition \ref{P: mainproof-generalcase}, we give a sketch of our proof. For a maximal $(Y,I)\in\mathcal{I}$ and $z\in Y$  such that $I$ is not torsion-low at $z$ and that  the  assumptions of Proposition \ref{P: mainproof-generalcase} are satisfied, we will find suitable maximal extension $(Y',I')$ of $(Y\setminus\{z\}, I)$. Either $(Y',I')$ is torsion-low at $z'\in Y'\setminus Y$, or $I'$ fixes more points that $I$ (Lemma \ref{L: mainproof-general case-1} and Lemma \ref{L: mainproof-plane with boundary more that 2 points}).  By induction, either there exists a maximal extension $(Y',I')$ of $(Y\setminus\{z\}, I)$ such that $I'$ is torsion-low at $z'\in Y'\setminus Y$, or there exists a  maximal extension $(Y',I')$ of $(Y\setminus\{z\}, I)$  such that $Y'\setminus Y$ contains infinitely many points (Lemma \ref{L: mainproof-general case-2}). In the second case, we  will prove that $Y'$ has a limit point $z'\in Y'\setminus Y$, and hence $I'$ is torsion-low at $z'$.

Now, we will prove Lemma \ref{L: mainproof-general case-1}, \ref{L: mainproof-plane with boundary more that 2 points}  and \ref{L: mainproof-general case-2} one by one, and then prove Proposition \ref{P: mainproof-generalcase}. We will use Lemma \ref{L: mainproof-general case-1} and \ref{L: mainproof-plane with boundary more that 2 points} when  proving Lemma \ref{L: mainproof-general case-2}, and  we will use Lemma  \ref{L: mainproof-general case-2} when  proving Proposition \ref{P: mainproof-generalcase}.

\begin{lemma}\label{L: mainproof-general case-1}
Let us suppose that  $(Y, I)$ is maximal in $(\mathcal{I},\precsim)$, that $I$ is not torsion-low at $z\in Y$, and  that the connected component of $M\setminus(Y\setminus\{z\})$ containing $z$ is neither homeomorphic to a sphere nor to a plane.  If for every  maximal extension $(Y', I')$ of $(Y\setminus\{z\},I)$ and every point $z'\in Y'\setminus (Y\setminus\{z\})$, $I'$ is not torsion-low at $z'$, then there exists a maximal extension  $(Y', I')\in \mathcal{I}$ of $(Y\setminus \{z\},I)$ such that $\# (Y'\setminus (Y\setminus\{z\}))>1$.
\end{lemma}

We will first give the idea of the proof of Lemma \ref{L: mainproof-general case-1}. Under the assumptions, we prove that there exists a fixed point $z_1$ of $f$ that is not fixed along $I$, and the trajectory of $z_1$ along $I$  is contractible in $M\setminus (Y\setminus \{z\})$. So, there exist isotopies that fix $Y\setminus\{z\}$ and $z_1$. We choose suitable $z_1$, and consider a maximal extension $(Y_1,I_1)$ of $(Y\setminus\{z\},I)$. We will prove either $Y_1\setminus Y$ contains at least two points, or $I_1$ is torsion-low at $z_1$.

\begin{proof}[Proof of Lemma \ref{L: mainproof-general case-1}]
 Fix a couple $(Y,I)$ maximal in $(\mathcal{I},\precsim)$ and $z_0\in Y$ satisfying the assumptions of this lemma. By Remark \ref{R: not torsion low means isolated}, $z_0$ is an isolated point of $Y$. Write  $Y_0=Y\setminus\{z_0\}$. Then $Y_0$ is a closed subset of $M$. Denote by $M_{Y_0}$ the connected component of $M\setminus Y_0$ containing $z_0$. By assumption, $M_{Y_0}$ is neither homeomorphic to a sphere nor to a plane.  Due to the definition of the torsion-low property and Proposition \ref{P: local rotation type non-isolated}, one has to consider the following four cases:
\begin{itemize}
\item[i)]$z_0$ is an isolated fixed point of $f$ and there exists a local isotopy $I'_{z_0}> I$ at $z_0$ which does not have a positive rotation type;
\item[ii)] $z_0$ is not an isolated fixed point of $f$ and there exists a local isotopy $I'_{z_0}> I$ at $z_0$ which has a negative rotation type;
\item[iii)]$z_0$ is an isolated fixed point of $f$ and there exists a local isotopy $I'_{z_0}<I$ at $z_0$ which does not have a negative rotation type;
\item[iv)] $z_0$ is  not an isolated fixed point of $f$ and  there exists a local isotopy $I'_{z_0}< I$ at $z_0$ which has a positive rotation type.
\end{itemize}
We will study the first two cases,  the other two can be treated in a similar way.

Let $\mathcal{F}_Y$ be a transverse foliation of $I$. In the case i), by Proposition \ref{P: local rotation type}, there exists a local isotopy $I_0$ at $z_0$ that is  torsion-low at $z_0$, and we know that $I<I'_{z_0}\lesssim I_0$,  so  $I$ has a negative rotation type at $z_0$; in the case ii), by Proposition \ref{P: local rotation type non-isolated}, we know that $I$ has a negative rotation type at $z_0$. In any case, $z_0$ is a source of $\mathcal {F}_Y$. We denote by $W$ the repelling basin of $z_0$ for $\mathcal{F}_Y$.

 Let $\pi_{Y_0}:\widetilde{M}_{Y_0}\rightarrow M_{Y_0}$ be the universal cover,  $\widetilde{I}=(\widetilde{f}_t)_{t\in[0,1]}$ be the identity isotopy that lifts  $I|_{M_{Y_0}}$, $\widetilde{f}=\widetilde{f}_1$ be the induced lift of $f|_{M_{Y_0}}$,  and $\widetilde{\mathcal{F}}$ be the lift of $\mathcal{F}_Y|_{M_{Y_0}}$.  Then, $\widetilde{I}$ fixes every point in $\pi_{Y_0}^{-1}\{z_0\}$, and every point in $\pi_{Y_0}^{-1}\{z_0\}$ is a source of $\widetilde{\mathcal{F}}$. We fix one element  $\widetilde{z}_0$ in $\pi_{Y_0}^{-1}\{z_0\}$, and denote by $\widetilde{W}$ the repelling basin of $\widetilde{z}_0$ for $\widetilde{\mathcal{F}}$.  Let $J_{\widetilde{z}_0}$ be an identity isotopy of the identity map of $\widetilde{M}_{Y_0}$ that fixes $\widetilde{z}_0$ and satisfies $\rho_s(J_{\widetilde{z}_0},\widetilde{z}_0)=\{1\}$. Let $\widetilde{I}^*$ be a maximal extension of $(\{\widetilde{z}_0\}, J_{\widetilde{z}_0}\widetilde{I})$, and $\widetilde{\mathcal{F}}^*$ be a transverse foliation of $\widetilde{I}^*$.

 Because $M_{Y_0}$ is neither homeomorphic to a sphere nor to a plane,  $\pi^{-1}_{Y_0}\{z_0\}$ is not reduced to one point, and $\widetilde{W}$ is a proper subset of $\widetilde{M}_{Y_0}$. So, there exists a proper\footnote{A leaf from $\infty$ to $\infty$.} leaf  of $\widetilde{\mathcal{F}}$.  Consequently, if we consider the end $\infty$ as a singularity, $\widetilde{f}$ can be blown up at $\infty$ by Proposition \ref{P: pre-blow-up}.
Note also that $\infty$ is accumulated by the points of $\pi_{Y_0}^{-1}\{z_0\}$, so $0$ belongs to $\rho_s(\widetilde{I},\infty)$.
Therefore, $\rho_s(\widetilde{I},\infty)$ is reduced to $0$ by the assertion vii) of Proposition \ref{P: pre-rotation set}, and $\rho_s(\widetilde{I}^*,\infty)$ is reduced to $-1$ by the first assertion of Proposition \ref{P: pre-rotation set}.

We can assert that $\widetilde{I}^*$ has   finitely many fixed points. We will prove it by contradiction.  Suppose that $\widetilde{I}^*$ fixes infinitely many  points.  Because $\rho_s(\widetilde{I}^*,\infty)$ is reduced to $-1$,  $\infty$ is not accumulated by fixed points of $\widetilde{I}^*$. Since $\widetilde{I}$ fixes each point in $\pi_{Y_0}^{-1}\{z_0\}$, $\widetilde{I}^*$ does not fix any point in $\pi_{Y_0}^{-1}\{z_0\}\setminus\{\widetilde{z}_0\}$. We know also that $\widetilde{z}_0$ is isolated in $\mathrm{Fix}(\widetilde{I}^*)$, because in case i) $z_0$ is isolated in $\mathrm{Fix}(f)$, and  in case ii) $\widetilde{I}^*$ has a negative rotation type by Proposition \ref{P: local rotation type non-isolated}.
 Therefore,  there exists a non-isolated  point $\widetilde{z}'$ in $\mathrm{Fix}(\widetilde{I}^*)$ such that $z'=\pi_{Y_0}(\widetilde{z}')\ne z_0$. Moreover, $z'$ is  a non-isolated fixed point of $f$. By Proposition \ref{P: pre-maximal implies no contractible fixed point}, there exists an extension $(Y',I')$ of $(Y_0, I)$  that fixes $z'$.  Let $\widetilde{I}'$ be the identity isotopy that  lifts  $I'|_{M_{Y_0}}$. Since $\pi_{Y_0}^{-1}(z')$ is included in $\mathrm{Fix}(\widetilde{I}')$, we have $\rho_s(\widetilde{I}',\infty)=0$. Therefore, $\widetilde{I}'$ and $ J^{-1}_{\widetilde{z}'} \widetilde{I}^*$ are equivalent as local isotopies at $\infty$, where $J_{\widetilde{z}'}$ is an identity isotopy of the identity map of $\widetilde{M}_{Y_0}$ that fixes $\widetilde{z}'$ and satisfies $\rho_s(J_{\widetilde{z}'},\widetilde{z}')=\{1\}$. Recall that $\pi_1(\mathrm{homeo}_0(\mathbb{R}^2,0))\cong\mathbb{Z}$ (see \cite{McCarty} or \cite{Hanmstrom-Homeotopygroups}), so $\widetilde{I}'$ and $ J^{-1}_{\widetilde{z}'} \widetilde{I}^*$ are also equivalent as local isotopies at  $\widetilde{z}'$. Recall Remark \ref{R: ac means torsion-low}, we get that $I'$ is torsion-low at $z'$, which contradicts with the assumption of this lemma.

Since $\rho_s(\widetilde{I}^*,\infty)$ is reduced to $-1$, the assertion v) of Proposition \ref{P: pre-rotation set} tells us that  $\infty$ is a source of $\widetilde{\mathcal{F}}^*$.

 We can  assert that $\widetilde{z}_0$ is not a sink of $\widetilde{\mathcal{F}}^*$. We will prove it in two cases. Recall that as  local isotopies at $\widetilde{z_0}$,  $\widetilde{I}^*$ is equivalent to $J_{\widetilde{z_0}}\widetilde{I}$.
 In  case i), as local isotopies at $\widetilde{z_0}$, both $\widetilde{I}^*$ and $J_{\widetilde{z_0}}\widetilde{I}$ are conjugate to a local isotopy at $z_0$ that is $\lesssim I'_{z_0}$, and  $I'_{z_0}$ does not have a positive rotation type, so $\widetilde{I}^*$ does not have a positive rotation type;
 in  case ii),  as local isotopies at $\widetilde{z_0}$, both $\widetilde{I}^*$ and $J_{\widetilde{z_0}}\widetilde{I}$ are conjugate to a local isotopy at $z_0$ that is $\lesssim I'_{z_0}$, and  $I'_{z_0}$ has a negative rotation type, so  $\widetilde{I}^*$ has a negative rotation type.

In $\widetilde{M}_{Y_0}\sqcup\{\infty\}$, there does not exist any closed leaf or oriented simple closed curve that consists of leaves and singularities of $\widetilde{\mathcal{F}}^*$ with the orientation inherited from the orientation of leaves. We can prove this assertion by contradiction. Let $\Gamma$ be  such a curve. Since $\infty$ is a source of $\widetilde{\mathcal{F}}^*$, it does not belong to $\Gamma$.  Let $U$ be the bounded component of $\widetilde{M}_{Y_0}\setminus \Gamma$, then $U$ contains the positive or the negative orbit of a wandering open set in $U\setminus\widetilde{f}(U)$ or $U\setminus\widetilde{f}^{-1}(U)$ respectively. This contradicts the area preserving property of $\widetilde{f}$.

Then, we can give a partial order $<$ over the set of singularities of $\widetilde{\mathcal{F}}^*$ such that $\widetilde{z}<\widetilde{z}'$ if there exists a leaf or a connection of leaves and singularities with the orientation inherited from the orientation of leaves from $\widetilde{z}'$ to $\widetilde{z}$. Since $\widetilde{\mathcal{F}}^*$ has only finitely many singularities, there exists a minimal singularity $\widetilde{z}_1$. Moreover, $\widetilde{\mathcal{F}}^*$ does not have any closed leaf  or a leaf from $\widetilde{z}_1$, and hence $\widetilde{z}_1$ is a sink of $\widetilde{\mathcal{F}}^*$.  Therefore, $\widetilde{f}$ fixes $\widetilde{z}_1$ and hence there exists a maximal extension $(Y_1, I_1)$ of $(Y_0, I)$ such that $Y_0\cup\{z_1\}\subset Y_1$, where $z_1=\pi_{Y_0}(\widetilde{z}_1)$.

Now, we will prove by contradiction that $Y_1\setminus  Y_0$ contains at least two points.  Suppose that  $Y_1= Y_0\sqcup\{z_1\}$.  Let $\mathcal{F}_{Y_1}$ be a transverse foliation of $I_1$, $\widetilde{I}_1$ be the identity isotopy that lifts $I_1|_{M_{Y_0}}$, and $\widetilde{\mathcal{F}}_1$ be the lift of $\mathcal{F}_{Y_1}|_{M_{Y_0}}$ to $\widetilde{M}_{Y_0}$. We know that $(Y_0,I)\sim(Y_0,I_1)$, so the lift of $f|_{M_{Y_0}}$ to $\widetilde{M}_{Y_0}$ associated to $I_1|_{M_{Y_0}}$ is also $\widetilde{f}$.  The set of singularities of $\widetilde{\mathcal{F}}_1$ is $\pi_{Y_0}^{-1}\{z_1\}$, and $\widetilde{z}_1$ is  an isolated singularity of $\widetilde{\mathcal{F}}_1$, so it is a sink, or a source, or a saddle of $\widetilde{\mathcal{F}_1}$ by  Remark \ref{R: pre-local dynamics of foliation-area preserving}.  We know that $\rho_s(\widetilde{I}^*,\infty)$ is reduced to $-1$ and  that $\rho_s(\widetilde{I}_1,\infty)$ is reduced to $0$, so $\widetilde{I}^*$ and $ J_{\widetilde{z}_1}\widetilde{I}_1$ are equivalent as  local isotopies at $\infty$, and are equivalent as local isotopies at $\widetilde{z}_1$. By the assumption, $I_1$ is not torsion-low at $z_1$, so $\widetilde{z}_1$ is a sink of $\widetilde{\mathcal{F}}_1$, and $z_1$ is a sink of $\mathcal{F}_{Y_1}$. Let $\widetilde{W}_1$ be the attracting basin of $\widetilde{z}_1$ for $\widetilde{\mathcal{F}}_1$. A leaf in $\partial\widetilde{W}_1$ is a proper leaf. For every fixed point $\widetilde{z}$ of $\widetilde{f}$, there exists a loop $\delta$ that is  homotopic to  its trajectory along $\widetilde{I}_1$ in $\widetilde{M}_{Y_0}\setminus\pi^{-1}_{Y_0}\{z_1\}$ (so in $\widetilde{M}_{Y_0}\setminus \{\widetilde{z_1}\}$) and is transverse to $\widetilde{\mathcal{F}}_1$. The linking number $L(\widetilde{I}_1,\widetilde{z},\widetilde{z}_1)$ is the index of the trajectory of $\widetilde{z}$ along $\widetilde{I}_1$ relatively to $\widetilde{z}_1$, so it is equal to the index of $\delta$ relatively to $\widetilde{z}_1$.  When $\widetilde{z}$ is in $\widetilde{W}_1$,  the loop $\delta$ is included in  $\widetilde{W}_1$ and is transverse to $\widetilde{\mathcal{F}}_1$, so $L(\widetilde{I}_1, \widetilde{z},\widetilde{z}_1)$ is positive; when $\widetilde{z}$ is not in $\widetilde{W}_1$, it is in one of the connected component of $\widetilde{M}_{Y_0}\setminus\overline{\widetilde{W}_1}$, and so is $\delta$, therefore $L(\widetilde{I}_1, \widetilde{z},\widetilde{z}_1)$ is equal to $0$.
 Since  $\widetilde{I}^*$ fixes $\widetilde{z}_0$ and $\widetilde{z}_1$,  the linking number $L(\widetilde{I}^*,\widetilde{z}_0,\widetilde{z}_1)$ is equal to $0$.  But  the restriction to $\widetilde{M}\setminus\{\widetilde{z}_1\}$ of $\widetilde{I}^*$ and $J_{\widetilde{z}_1}\widetilde{I}_1$ are homotopic, so we have
\[ L(\widetilde{I}_1, \widetilde{z}_0,\widetilde{z}_1)=L(\widetilde{I}^*,\widetilde{z}_0,\widetilde{z}_1)-1=-1,\]
and find a contradiction.
\end{proof}

\begin{lemma}\label{L: mainproof-plane with boundary more that 2 points}
Suppose that $(Y, I)$ is maximal in $\mathcal{I}$, that  $I$ is not torsion-low at $z\in Y$,   that the connected component of $M\setminus(Y\setminus\{z\})$ containing $z$ is homeomorphic to a plane and its  boundary in $M$ contains at least two points of $Y\setminus\{z\}$. If for every  maximal extension $(Y', I')$ of $(Y\setminus\{z\},I)$ and every $z'\in Y'\setminus (Y\setminus\{z\})$, $I'$ is not torsion-low at $z'$, then there exists a maximal extension  $(Y', I')\in \mathcal{I}$ of $(Y\setminus \{z\},I)$ such that $\# (Y'\setminus (Y\setminus\{z\}))>1$.
\end{lemma}

The idea of the proof of this lemma is similar to that of Lemma \ref{L: mainproof-general case-1}.
The difference is the following: in the assumption of Lemma \ref{L: mainproof-general case-1},  the connected component of $M\setminus(Y\setminus\{z\})$ containing $z$ is neither homeomorphic to a sphere nor to a plane;
but in the assumption of this lemma,  the connected component of $M\setminus(Y\setminus\{z\})$ containing $z$ is homeomorphic to a plane. So we do not lift the objects to the universal covering space, but we can use different techniques and get the same results.

\begin{proof}[Proof of Lemma \ref{L: mainproof-plane with boundary more that 2 points}]
 Fix a maximal $(Y,I)\in\mathcal{I}$ and $z_0\in Y$ satisfying the assumptions of this lemma. Write $Y_0=Y\setminus \{z_0\}$, and denote by $M_{Y_0}$ the connected component of $M\setminus Y_0$ containing $z_0$. Then  $M_{Y_0}$ is homeomorphic to a plane and $\#\partial M_{Y_0}> 1$. As in the proof of Lemma \ref{L: mainproof-general case-1}, since $I$ is not torsion-low at $z_0$, one has to consider the following four cases:
\begin{itemize}
\item[i)]$z_0$ is an isolated fixed point of $f$ and there exists a local isotopy $I'_{z_0}> I$ at $z_0$ which does not have a positive rotation type;
\item[ii)] $z_0$ is not an isolated fixed point of $f$ and there exists a local isotopy $I'_{z_0}> I$ at $z_0$ which has a negative rotation type;
\item[iii)]$z_0$ is an isolated fixed point of $f$ and there exists a local isotopy $I'_{z_0}<I$ at $z_0$ which does not have a negative rotation type;
\item[iv)] $z_0$ is  not an isolated fixed point of $f$ and  there exists a local isotopy $I'_{z_0}< I$ at $z_0$ which has a positive rotation type.
\end{itemize}
As in the proof of Lemma \ref{L: mainproof-general case-1}, we only study the first two cases.

Let $\mathcal{F}_Y$ be a transverse foliation of $I$. As in the proof of Lemma \ref{L: mainproof-general case-1}, we know that  $z_0$ is a source of $\mathcal{F}_Y$.

Recall that $M_{Y_0}$ is homeomorphic to a plane, and $I$ fixes $\partial M_{Y_0}$ and $z_0$. Near the end $\infty$, $I|_{M_{Y_0}}$ can be viewed as a local isotopy at $\infty$.
 When $M$ is homeomorphic to the sphere or the plane, by Proposition \ref{P: zero prime end rotation number}, we know that $I|_{M_{Y_0}}$ can be blown up at $\infty$ and the blow-up rotation number  $\rho(I|_{M_{Y_0}},\infty)$, that was defined in Section \ref{S: pre-blow up},  is equal to $0$; in the other case, by conjugate $I|_{M_{Y_0}}$ to its lift to the universal covering space of $M$, we can still apply Proposition \ref{P: zero prime end rotation number} and get the same result.

Let $I^*$ be a maximal extension of $(\{z_0\},J_{z_0}I|_{M_{Y_0}})$, where  $J_{z_0}$ is an identity isotopy of the identity map of $M_{Y_0}$ that fixes $z_0$ and satisfies $\rho_s(J_{z_0},z_0)=\{1\}$. Let $\mathcal{F}^*$ be a transverse foliation of $I^*$. In  cases i) and  ii), by the same argument in the proof of Lemma \ref{L: mainproof-general case-1}, we know that $z_0$ is not a sink of $\mathcal{F}^*$.

 We can assert that $\infty$ is a source of $\mathcal{F}^*$.  Indeed, when the total area of $M_{Y_0}$ is finite, $f|_{M_{Y_0}}$ is area preserving as a local homeomorphism at $\infty$, referring to Section \ref{S: pre-local rotation set}, we know that  $\rho_s(I|_{M_{Y_0}},\infty)$ is not empty and is reduced to $0$ by the assertion vii) of Proposition \ref{P: pre-rotation set}. Then, by the assertion i) of Proposition \ref{P: pre-rotation set}, $\rho_s(I^*,\infty)$ is reduced to $-1$,  and by the assertion v) of Proposition \ref{P: pre-rotation set}, $\infty$ is a source of $\mathcal{F}^*$.  However, the total area of $M_{Y_0}$ may be infinite. In this case, we can not get the result that $\rho_s(I|_{M_{Y_0}},\infty)$ is not empty.  But in any case, we can prove the assertion by  considering the following two cases:
 \begin{itemize}
 \item[-] Suppose that  $\rho_s(I|_{M_{Y_0}},\infty)$ is not empty. As in the case where the total area of $M_{Y_0}$ is finite,  $\rho_s(I|_{M_{Y_0}},\infty)$ is  reduced to $0$, and   $\rho_s(I^*,\infty)$ is reduced to $-1$. Therefore, $\infty$ is a source of $\mathcal{F}^*$ by the assertion v) of Proposition  \ref{P: pre-rotation set}.
 \item[-] Suppose that $\rho_s(I|_{M_{Y_0}},\infty)$ is  empty. Since  $f|_{M_{Y_0}}$ is area preserving, $f|_{M_{Y_0}}$ is not conjugate to a contraction or an expansion at $\infty$. Referring to Section \ref{S: pre-local rotation set}, we know that  the germ of $f|_{M_{Y_0}}$ at $\infty$ is conjugate to a local homeomorphism $z\mapsto e^{i2\pi \frac{p}{q}}z(1+z^{qr})$ at $0$ with $q,r\in\mathbb{N}^+$ and $p\in\mathbb{Z}$.  The latter local homeomorphism can be blown up at $0$, and the extension at the added circle is the counter-clockwise rotation through an angle $2\pi p/q$. Recall that $\rho(I|_{M_{Y_0}},\infty)=0$,  we can deduce that $p\in q\mathbb{Z}$, and hence the Lefschetz index $i(f|_{M_{Y_0}},\infty)>1$. By the first assertion of Proposition \ref{P: pre-index and rotation type}, there exists a local isotopies $I_0$ of $f$ at $\infty$ such that for all locally transverse foliation $\mathcal{F}_0$ of $I_0$, we have $i(\mathcal{F}_0,\infty)=i(f,\infty)>1$. Moreover, by Proposition \ref{P: pre-dynamics of a foliation}, we know that $\infty$ is a petal or a mixed type singularity for  $\mathcal{F}_0$. Then, by Proposition \ref{P: pre-blow-up}, the blow-up rotation number $\rho(I_0,\infty)=0$, and hence $I|_{M_{Y_0}}$ and $I_0$ are equivalent as local isotopy at $\infty$. By the second assertion of Proposition \ref{P: pre-index and rotation type}, $\infty$ is a source of $\mathcal{F}^*$.
\end{itemize}

 Then, like in the proof of Lemma \ref{L: mainproof-general case-1}, we deduce that $I^*$ fixes  finitely many points, that there exists a sink $z_1$ of $\mathcal{F}^*$, and that there exists a maximal extension $(Y',I')\in\mathcal{I}$ of $(Y_0,I)$ such that $Y_0\cup\{z_1\}\subset Y'$  and  $\#(Y'\setminus Y_0) >1$.
\end{proof}

The following lemma is a  consequence of the previous two lemmas.

\begin{lemma}\label{L: mainproof-general case-2}
Let us suppose that  $(Y, I)$ is maximal in $(\mathcal{I},\precsim)$, that $I$ is not torsion-low at $z\in Y$.
When $M$ is a sphere, we suppose that $Y\setminus \{z\}$ contains at least two points; and when $M$ is a plane, we suppose that $Y\setminus\{z\}$ is not empty. If for every  maximal extension $(Y', I')$ of $(Y\setminus\{z\},I)$ and every point $z'\in Y'\setminus (Y\setminus\{z\})$, $I'$ is not torsion-low at $z'$, then there exists a maximal extension  $(Y'', I'')\in \mathcal{I}$ of $(Y\setminus \{z\},I)$ such that $\# (Y''\setminus (Y\setminus\{z\}))=\infty$.
\end{lemma}

\begin{proof}
  Fix a couple $(Y,I)$ maximal in $(\mathcal{I},\precsim)$ and $z\in Y$ satisfying the assumptions of the lemma. By the previous two lemmas, there exists a maximal extension $(Y_1, I_1)$ of $(Y\setminus\{z\}, I)$ such that $\#(Y_1\setminus (Y\setminus\{z\}))>1$. If $\#(Y_1\setminus (Y\setminus\{z\}))=\infty$, the proof is finished;  otherwise, we fix a point $z_1\in Y_1\setminus (Y\setminus \{z\})$. By hypothesis, $I_1$ is not torsion-low at $z_1$; the connected component of $M\setminus (Y_1\setminus\{z_1\})$ containing $z_1$ is not homeomorphic to a sphere; and when the connected component of $M\setminus (Y_1\setminus\{z_1\})$ containing $z_1$  is homeomorphic to a plane, its boundary in $M$ contains at least two points.
  Since a  maximal extension of $(Y_1\setminus\{z_1\},I_1)$  is also a maximal extension of $(Y\setminus\{z\}, I)$,  the couple $(Y_1, I_1)$ and $z_1\in Y_1$ satisfies the assumptions of the previous two lemmas. We apply the previous two lemmas, and deduce that  there exists a maximal extension $(Y_2,I_2)\in\mathcal{I}$ of $(Y_1\setminus\{z_1\}, I_1)$ such that $\#(Y_2\setminus (Y_1\setminus\{z_1\}))>1$.  If $\#(Y_2\setminus (Y_1\setminus\{z_1\}))=\infty$, the proof is finished;  if $\#(Y_2\setminus (Y_1\setminus\{z_1\}))<\infty$, we continue the construction\ldots

  Then, either we end the proof in finite steps, or we can construct a strictly increasing sequence
\[(Y\setminus\{z\}, I)\prec(Y_1\setminus\{z_1\}, I_1)\prec (Y_2\setminus\{z_2\},I_2)\prec (Y_3\setminus\{z_3\},I_3)\cdots\]
By Proposition \ref{P: pre-exist uper bound of  isotopies sequence}, there exists an upper bound $(Y_{\infty},I_{\infty})\in\mathcal{I}$ of this sequence, where  $Y_{\infty}=\overline{\cup_{n\ge 1} (Y_n\setminus\{z_n\})}$. By Theorem \ref{T: pre-exist max isotopy}, there exists a maximal extension $(Y', I')\in\mathcal{I}$ of $(Y_{\infty},I_{\infty})$. It is also a maximal extension of  $(Y\setminus\{z\}, I)$, and satisfies $\# (Y'\setminus (Y\setminus\{z\}))=\infty$.
\end{proof}

\begin{proof}[Proof of Proposition \ref{P: mainproof-generalcase}]
We will prove this proposition by contradiction. Fix a maximal element $(Y,I)\in\mathcal{I}$ and $z_0\in Y$ such that $I$ is not torsion-low at $z_0$. When $M$ is a sphere, we suppose that $Y\setminus \{z\}$ contains at least two points; and when $M$ is a plane, we suppose that $Y\setminus\{z\}$ is not empty.
Write $Y_0=Y\setminus\{z_0\}$, and suppose that for all  maximal extension $(Y', I')$ of $(Y_0,I)$ and $z'\in Y'\setminus Y_0$, $I'$ is not torsion-low at $z'$. By the previous lemma, there exists a maximal extension  $(Y', I')$ of $(Y_0,I)$ such that $\# (Y'\setminus Y_0)=\infty$.

 Denote by $M_{Y_0}$ the connected component of $M\setminus Y_0$ containing $z_0$.  Let us prove by contradiction that $Y'\setminus Y_0\subset M_{Y_0}$.  Suppose that there exists $z_1\in Y'\setminus Y_0$ that is in another component of $M\setminus Y_0$. Since both $(Y,I)$ and $(Y',I')$ are extensions of $(Y_0,I)$, the trajectory of $z_1$ along $I$ and $I'$ are  homotopic in $M\setminus Y_0$.  So,  the trajectory of $z_1$ along $I$  is homotopic to zero in $M\setminus Y_0$. When $z_1\notin M_{Y_0}$, its trajectory along $I$ is in another component of $M\setminus Y_0$, this trajectory is homotopic to zero in $M\setminus Y$, which contradicts the maximality of $(Y,I)$ by Proposition \ref{P: pre-maximal implies no contractible fixed point}.

Then, one has to consider two cases:
 \begin{itemize}
 \item[-]$M_{Y_0}$ is neither homeomorphic to a sphere nor to a plane,
   \item[-]$M_{Y_0}$ is homeomorphic to a plane, and its boundary in $M$ contains at least two points.
   \end{itemize}

We will proof the proposition in the first case. In the second case, the proof is almost the same except that we will not lift the isotopies and the foliations to the universal cover, because $M_{Y_0}$ itself is homeomorphic to a plane.

Suppose that $M_{Y_0}$ is neither homeomorphic to a sphere nor to a plane.    Let $\pi_{Y_0}:\widetilde{M}_{Y_0}\rightarrow M_{Y_0}$ be the universal cover,  $\widetilde{I}$  the identity isotopy that   lifts   $I|_{M_{Y_0}}$, $\widetilde{I}'$  the identity isotopy that   lifts   $I'|_{M_{Y_0}}$, and $\widetilde{f}$  the lift of $f|_{M_{Y_0}}$ associated to $I|_{M_{Y_0}}$.  Since both $I$ and $I'$ are maximal, the point $z_0$ does not belong to  $Y'$. Moreover, $\widetilde{f}$ is also the lift of  $f|_{M_{Y_0}}$ associated to $I'|_{M_{Y_0}}$. In particular, $\widetilde{f}$ fixes every point in $\pi_{Y_0}^{-1}(\{z_0\}\cup Y'\setminus Y_0)$.
 Fix  $\widetilde{z}_0\in \pi_{Y_0}^{-1}\{z_0\}$.

 \begin{sublem}
For every $z\in Y'\setminus Y_0$, there exists $\widetilde{z}\in\pi_{Y_0}^{-1}\{z\}$ such  that $\widetilde{z}_0$ and $\widetilde{z}$ are linked relatively to  $\widetilde{I}$.
\end{sublem}
\begin{proof}
  Let $\mathcal{F}$ be a transverse foliation of $I$, and $\widetilde{\mathcal{F}}$ be the lift of $\mathcal{F}|_{M_{Y_0}}$ to $\widetilde{M}_{Y_0}$. Fix  $z\in Y'\setminus Y_0$ and $\widetilde{z}\in\pi_{Y_0}^{-1}\{z\}$.  Since $I$ is a maximal isotopy, the trajectory of $\widetilde{z}$ along $\widetilde{I}$ is a loop that is not homotopic to zero in $\widetilde{M}_{Y_0}\setminus \pi_{Y_0}^{-1}\{z_0\}$. Let $\delta$ be a loop that  is positively transverse to $\widetilde{\mathcal{F}}$, and is homotopic to the trajectory of $\widetilde{z}$ along $\widetilde{I}$ in $\widetilde{M}_{Y_0}\setminus \pi_{Y_0}^{-1}\{z_0\}$. By choosing suitable $\delta$, we can suppose that $\delta$ intersects itself at most finite times,  that each intersection point is a double point, and that the intersections are transverse. So, $\widetilde{M}_{Y_0}\setminus\delta$ has finitely many components, and we can define a locally constant function $\Lambda:\widetilde{M}_{Y_0}\setminus\delta\rightarrow \mathbb{Z}$ such that
 \begin{itemize}
 \item[-] $\Lambda$ is equal to $0$ in the component of $\widetilde{M}_{Y_0}\setminus\delta$ that is not relatively compact;
 \item[-] $\Lambda(\widetilde{z}')-\Lambda(\widetilde{z}'')$ is equal to the (algebraic) intersection number of $\delta$ and any arc from $\widetilde{z}''$ to $\widetilde{z}'$.
 \end{itemize}
 This function is not constant, and we have either $\max \Lambda>0$ or $\min \Lambda<0$. Suppose that we are in the first case (the other case can be treated similarly). Let $U$ be a component of $\widetilde{M}_{Y_0}\setminus\delta$ such that $\Lambda$ is equal to $\max \Lambda>0$ in $U$.
\begin{figure}[h]
 \center
 \includegraphics[width=4cm]{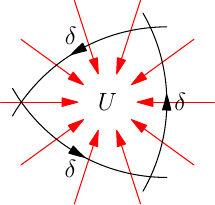}
 \caption{}
 \label{fig: dual-function}
 \end{figure}
  As in Figure \ref{fig: dual-function}, the boundary of $U$ is a  sub-curve of $\delta$ with the orientation such that $U$ is to the left of its boundary, and is also transverse to $\widetilde{\mathcal{F}}$. So, there exists a singularity of $\widetilde{\mathcal{F}}$ in $U$. Note the fact that the set of singularities of $\widetilde{\mathcal{F}}$ is $\mathrm{Fix}(\widetilde{I})=\pi^{-1}_{Y_0}\{z_0\}$. So, there exists  an automorphism $T$ of the universal cover space such that $T(\widetilde{z}_0)$ belongs to $U$, and the index of $\delta$ relatively to $T(\widetilde{z}_0)$ is positive. Note also that the linking number $L(\widetilde{I},\widetilde{z}, T(\widetilde{z}_0))$ is equal to the index of $\delta$ relatively to $T(\widetilde{z}_0)$, and hence is equal to $\Lambda(T(\widetilde{z}_0))$ by definition of $\Lambda$. So, $T(\widetilde{z}_0)$ and $\widetilde{z}$ are  linked relatively to $\widetilde{I}$. Consequently, $\widetilde{z}_0$ and $T^{-1}(\widetilde{z})$ are linked relatively to $\widetilde{I}$.
 \end{proof}

As in the proof of Lemma \ref{L: mainproof-general case-1}, we know that $\widetilde{f}$ can be blown up at $\infty$. Since $\infty$ is accumulated by both the points in $\pi_{Y_0}^{-1}\{z_0\}$ and the points in $\pi_{Y_0}^{-1}(Y'\setminus Y_0)$, both $\rho_s(\widetilde{I},\infty)$ and $\rho_s(\widetilde{I}',\infty)$ contain $0$. By the assertion vii) of  Proposition \ref{P: pre-rotation set}, both $\rho_s(\widetilde{I},\infty)$ and $\rho_s(\widetilde{I}',\infty)$ are reduced to $0$, so $\widetilde{I}$ and $\widetilde{I'}$  are equivalent as local isotopies at $\infty$.
Therefore, for every point $z\in Y'\setminus Y_0$, there exists $\widetilde{z}\in\pi_{Y_0}^{-1}\{z\}$ such  that $\widetilde{z}_0$ and $\widetilde{z}$ are linked relatively to  $\widetilde{I}'$. Let us denote by $L$ the set of points $\widetilde{z}\in\pi_{Y_0}^{-1}(Y'\setminus Y_0)$ such that $\widetilde{z}$ and $\widetilde{z}_0$  are linked relatively to $\widetilde{I}'$. It  contains infinitely many points.

Let $\gamma$ be the trajectory of $\widetilde{z}_0$ along the isotopy $\widetilde{I}'$, and $V$ be the connected component of $\widetilde{M}_{Y_0}\setminus \gamma$ containing $\infty$. Then $K=\widetilde{M}_{Y_0}\setminus V$ is a compact set that contains all the  fixed points of $\widetilde{I}'$ that are linked with $\widetilde{z}_0$ relatively to $\widetilde{I}'$. In particular, $L\subset K$. Then, there exists  $\widetilde{z}'\in K$ that is accumulated by points of $L$. We know that $\mathrm{Fix}(\widetilde{I}')$ is a closed set. So, $\widetilde{z}'$ belongs to $\mathrm{Fix}(\widetilde{I}')=\pi^{-1}(Y'\setminus Y_0)$. We find a point  $\widetilde{z}'$  that is not isolated in $\pi_{Y_0}^{-1}(Y'\setminus Y_0)$, and  a point $z'=\pi_Y(\widetilde{z}')$  that is not isolated in $Y'$. By Remark \ref{R: ac means torsion-low},  $I'$ is torsion-low at $z'$. We get a contradiction.
\end{proof}

\begin{proof}[Proof of Proposition \ref{P: mainproof-plane}]
We only need to prove that there exists $(X,I)\in\mathcal{I}_0$ such that $X\ne \emptyset$, because one knows $(\emptyset, I_0)\precsim (X,I)$ for all isotopy $I_0$ from the identity to $f$ and all $(X,I)\in\mathcal{I}$ when $M$ is a plane.

One has to consider the following  three cases:
\begin{itemize}
\item[-] Suppose that $\mathrm{Fix}(f)$ is reduced to one point $z_0$. In this case,  similarly to the proof  of Proposition \ref{P: local rotation type}, we can find an isotopy $I_0$ that fixes $z_0$ and is torsion-low at $z_0$. Then, $(\{z_0\}, I_0)$ belongs to $\mathcal{I}_0$.

\item[-] Suppose that there exists a connected component $X$ of $\mathrm{Fix}(f)$ that is not reduced to one point. By Proposition \ref{P: pre-connected implies unlinked in plane},  there exists a maximal isotopy $I$ of $f$ that fixes all the points in $X$. Because  every $z\in X$ is not isolated in $X$, $I$ is torsion-low at $x$. So, $(X,I)$ belongs to $\mathcal{I}_0$.

\item[-] Suppose that $\mathrm{Fix}(f)$ is totally disconnected and  contains at least two points. In this case, there exists a maximal $(Y,I)\in\mathcal{I}$ such that $\#Y\ge 2$. If $I$ is torsion-low at a point in $Y$,  the proof is finished; if $I$ is not torsion-low at every $z\in Y$, we fix $z_0\in Y$ and can find a maximal extension $(Y',I')$ of $(Y\setminus\{z_0\}, I)$ and $z'\in Y'\setminus (Y\setminus\{z_0\})$ such that $I'$ is torsion-low at $z'$ by Proposition \ref{P: mainproof-generalcase}. Consequently, $(\{z'\}, I')$ belongs to $\mathcal{I}_0$. \qedhere
\end{itemize}
\end{proof}

\begin{proof}[Proof of Proposition \ref{P: mainproof-sphere}]
  One knows $(X,I)\precsim(Y,I')$ for all $(Y,I')\in\mathcal{I}$ satisfying $X\subset Y$, when $M$ is a sphere and $\#X\le 1$. So, we only need to prove the following two facts:
\begin{itemize}
\item[i)] there exists $(X,I)\in\mathcal{I}_0$ such that $X\ne \emptyset$;
\item[ii)] given $(X,I)\in \mathcal{I}_0$ such that $\# X=1$, there exists $(X',I')\in\mathcal{I}_0$ such that $X\subsetneqq X'$.
\end{itemize}
One has to consider the following two cases:
\begin{itemize}
\item[-] Suppose that $\#\mathrm{Fix}(f)=2$.  In this case, we will prove that there exists an identity isotopy that fixes both fixed points and is torsion-low at each fixed point, which implies both i) and ii).

    Denote by $N$ and $S$ the two fixed points. Since both $N$ and $S$ are isolated fixed points, we can find an identity isotopy $I$ that fixes both $N$ and $S$ and is torsion-low at $S$. We will prove that $I$ is also torsion-low at $N$.

     Let $J_N$ (resp. $J_S$) be an identity isotopy of the identity map of the sphere that fixes both $N$ and $S$ and satisfies $\rho_s(J_N, N)=\{1\}$ (resp. $\rho_s(J_S,S)=\{1\}$). One knows that  $J_N$ and $J^{-1}_S$ are homotopic relatively to $\{N,S\}$.

     For every $k\ge 1$, since $I$ is torsion-low at $S$, $J_S^{-k}I$ has a negative rotation type as a local isotopy at $S$. Let $\mathcal{F}_k$ be a transverse foliation of $J_S^{-k}I$. By the first statement of Proposition \ref{P: local rotation type}, one knows that  $S$ is a source of $\mathcal{F}_k$. Since $f$ is area preserving and $\mathcal{F}_k$ has exactly two singularities, $N$ is a sink of $\mathcal{F}_k$. Note the fact that the restrictions to $M\setminus\{S,N\}$ of $J_N^kI$ and $ J_S^{-k}I$ are homotopic. So, $J_N^kI$ has a positive rotation type as a local isotopy at $N$.

     Similarly, for every $k\ge 1$, $J_N^{-k}I$ has a negative rotation type as a local isotopy at $N$. Therefore, $I$ is torsion-low at $N$.

\item[-] Suppose that $\#\mathrm{Fix}(f)\ge 3$.

 We can prove i) by a similar discussion to the second part and the third part of the proof of Proposition \ref{P: mainproof-plane}. Alternatively, we  give the following direct proof.  Fix a maximal $(Y,I)\in\mathcal{I}$ such that $\# Y\ge 3$. If $Y$ is infinite, there exists a point $z\in Y$ that is not isolated in $Y$, and hence $I$ is torsion-low at $z$. If $Y$ is finite, we consider a transverse foliation of $I$ and know that there is a saddle singulary point $z$ of $\mathcal{F}$ by the Poincar\'e-Hopf formula, Proposition \ref{P: pre-dynamics of a foliation} and Remark \ref{R: pre-local dynamics of foliation-area preserving}, so $I$ is torsion-low at $z$. In both cases, there exists $z\in Y$ such that $(\{z\}, I)\in\mathcal{I}_0$.

 To prove ii), we fix $(X,I)\in \mathcal{I}_0$ such that $X$ contains only one point, and denote this fixed point by $S$.
 For a maximal extension $(Y,I')$ of $(X,I)$ that is torsion-low at $S$, if $I'$ is torsion-low at another fixed point, we get the result; if $I'$ fixes at least three fixed point and is not torsion-low at any fixed point except $S$, we apply  Proposition \ref{P: mainproof-generalcase} and can  find another maximal extension of $(X,I)$ that is torsion-low at another fixed point and is equivalent to $I'$ as local isotopies at $S$. So, we only need to prove that there exists a maximal extension $(Y,I')$ of $(X,I)$ that is torsion-low at $S$ and satisfies one of the following two condition: $I'$ is torsion-low at another fixed point or $\# Y\ge 3$.

 Fix  a maximal extension $(Y,I')\in \mathcal{I}$  of $(X,I)$ such that $\rho_s(I',S)=\rho_s(I,S)$. Of course, $I'$ is torsion-low at $S$. If $\#(Y\setminus X)\ge 2$, we finish the proof.
 Now, we suppose that $Y=\{S,N\}$. By the maximality of $(Y,I')$, $S$ is not accumulated by contractible fixed points of $f$ associated to $I'$, and hence $I'$ has either a positive or a negative rotation type at $S$.
 We suppose that $I'$ has a positive rotation type, the proof in the other case is similar.

 Let $J_S$ be the isotopy of the identity that fixes $S$ and $N$ and satisfies $\rho_s(J_S,S)=\{1\}$. We consider a maximal extension $I''$ of $(Y,J_S^{-1}I')$. Since $I'$ is torsion-low at $S$, $S$ is either accumulated by contractible fixed points of $f$ associated to $I''$ or has a negative rotation type.  In the first case, $I''$ is torsion-low at $S$ and satisfies $\#\mathrm{Fix}(I'')\ge 3$; in the second case, $I''$ is still torsion-low at $S$. In both case, $I''$ is torsion-low at $S$. If $\#\mathrm{Fix}(I'')\ge 3$, we finish the proof.

Now we suppose that $I''$ has a negative rotation type at $S$ and $\mathrm{Fix}(I'')=Y=\{N,S\}$, we will finish the proof by proving that both $I'$ and $I''$ is torsion-low at $N$. Let $\mathcal {F}'$ be a transverse foliation of $I'$ and $\mathcal{F}''$ a transverse foliation of $I''$. Then, $S$ is a sink of $\mathcal{F}'$ and is a source of $\mathcal{F}''$. So, $N$ is a source of $\mathcal{F}'$ and is a sink of $\mathcal{F}''$. Therefore, $I'$ has a negative rotation type at $N$, $I''$ has a positive rotation type at $N$. Recall that $I''$ is a maximal extension of $(Y,J_S^{-1}I')\sim(Y,J_N I')$, where $J_N$ be the isotopy of the identity that fixes $S$ and $N$ and satisfies $\rho_s(J_N,N)=\{1\}$. So, both $I'$ and $I''$ are torsion-low at $N$.\qedhere
\end{itemize}
\end{proof}

\begin{remark}
In our proof, we consider first a maximal extension $I'$ of $(X,I)$ that satisfies $\rho_s(I',S)=\rho_s(I,s)$.  When $I'$ is not a ``good" one, we construct another identity isotopy $I''$ to get the result. Even though $\rho_s(I, S)$ and $\rho_s(I'',S)$ are different, $I''$ is still an extension of $(X,I)$ because $M$ is a sphere and $X$ is reduced to a single point. It gives an explanation why we exclude the case that $M$ is a sphere and $X$ is reduced to a single point in Remark \ref{R: globle torsion-low 1}.
\end{remark}

\bigskip

We finish the section by proving Proposition \ref{P: torsion-low isotopy has most fixed points}.
\begin{proof}[Proof of Proposition \ref{P: torsion-low isotopy has most fixed points}]
Let $f$ be an area preserving homeomorphism of $M$ that is isotopic to the identity and has finitely many fixed points. When $\mathrm{Fix}(f)$ is  empty, the proposition is trivial. So, we suppose that $\mathrm{Fix}(f)$ is not empty. Let
 \[n=\max\{\#\mathrm{Fix}(I): I \text{ is an identity isotopy of } f\}.\]
One has to consider the following three cases:
\begin{itemize}
\item[-]  Suppose that $M$ is a plane and $f$ has exactly one fixed point. As  in the first part of the proof of Proposition \ref{P: mainproof-plane}, there exists  an identity isotopy  that fixes this fixed point and is torsion-low at this fixed point.
\item[-] Suppose that $M$ is a sphere and $f$ has exactly two fixed points. As  in the first part of the proof of Proposition \ref{P: mainproof-sphere}, there  exists  an identity isotopy that fixes these two fixed points and  is torsion-low at each fixed point.
\item[-] Suppose that we are not in the previous two cases. Let $\mathfrak{I}$ be the set of identity isotopies of $f$ with $n$ fixed points. It is not empty.  We can give a preorder $\lhd$ over $\mathfrak{I}$ such that $I\lhd I'$  if and only if
   \[\#\{z\in\mathrm{Fix}(I), I \text{ is torsion-low at } z\}\le\#\{z\in\mathrm{Fix}(I'), I' \text{ is torsion-low at } z\}.\]
    Since $\#\{z\in\mathrm{Fix}(I), I \text{ is torsion-low at } z\}$ is not larger than $n$ for all $I\in\mathfrak{I}$,  $\mathfrak{I}$  has a maximal element. Fix a maximal element $I$ of  $\mathfrak{I}$. We will prove by contradiction that $I$ is torsion-low at every $z\in\mathrm{Fix}(I)$.

    Suppose that $I$ is not torsion-low at $z_0\in\mathrm{Fix}(I)$. Write $Y_0=\mathrm{Fix}(I)\setminus\{z_0\}$. By Proposition \ref{P: mainproof-generalcase}, there exist a maximal extension $I'$ of $(Y_0, I)$ and $z'\in\mathrm{Fix}(I')\setminus Y_0$ such that $I'$ is torsion-low at $z'$. This  contradicts with the maximality of   $I$ in $(\mathfrak{J},\lhd)$. \qedhere
\end{itemize}
\end{proof}

\section{Examples}\label{S: examples}

\begin{example}\label{Ex: local rotation set is infty}(An orientation and area  preserving  homeomorphism whose local rotation set is reduced to $\infty$)

Let $f$ be the  homeomorphism of $\mathbb{C}$ defined by
\[f(z)=\left\{\begin{array}{ll}0 &\text{ for } z=0,\\
z e^{i2\pi/|z| } &\text{ for } z\ne 0.\end{array}\right.\]
It is area preserving and fixes $0$. Moreover, $\rho_s(I, 0)$ is reduced to $+\infty$ for every  isotopy $I$ of $f$ fixing $0$.
\end{example}

\begin{example}\label{Ex: area preserving condition is necessay}(Example of Remark \ref{R: area preserving is neccesary for the existence of globle torsion-low isotopy})

We will construct an orientation preserving diffeomorphism $f$ of the sphere with $2$ fixed points such  that $f$ is area preserving in a neighborhood of each fixed point but there  exists no torsion-low maximal isotopy of $f$.

Let $\varphi$ be a diffeomorphism of $[0,1]$  that satisfies
\[\left\{\begin{array}{ll}\varphi(y)=y &\text{ for } y\in [0,1/6]\cup[5/6,1],\\
  \varphi(y)<y &\text{ for } y\in(1/6, 5/6).\end{array}\right.\]
 Let $g$ be a diffeomorphism of $\mathbb{R}\times[0,1]$ that is defined by
 \[g(x,y)=(x+3y,\varphi(y)).\]
  We define an equivalence relation $\sim$ on $\mathbb{R}\times[0,1]$ such that
 \[\left\{\begin{array}{ll}(x,y)\sim(x+1,y) &\text{ for all } (x,y)\in \mathbb{R}\times(0,1),\\
  (x,0)\sim(x',0) &\text{ for all } x,x'\in\mathbb{R},\\
   (x,1)\sim(x',1) &\text{ for all } x,x'\in\mathbb{R}.\end{array}\right.\]
   Then, $\mathbb{R}\times[0,1]/_{\sim}$ is a sphere, and $g$ descends to a diffeomorphism $f$ of the sphere that has  two fixed points and is area preserving near each fixed point. Note the facts that  every maximal isotopy $I$ fixes both fixed points of $f$, that the rotation number of $I$ at each fixed point is an integer, and that the sum of the rotation numbers of $I$ at both fixed point is $3$. By the last statement of Proposition \ref{P: rotation set of a  torsion-low local isotopy}, there does not exist any torsion-low maximal isotopy of $f$.
\end{example}

\begin{example}(Example of Remark \ref{R: inegality is not strict}) \label{Ex: inequality not strict}

We will construct an orientation and area preserving diffeomorphism of the sphere such that there does not exist any maximal isotopy $I$ satisfying
\[-1<\rho(I,z)<1, \text{ for every } z\in \mathrm{Fix}(I).\]

Let $g$ be a diffeomorphism of $\mathbb{R}\times[0,1]$ that is defined by
\[g(x,y)=(x+y,y).\]
We define an equivalence relation $\sim$ on $\mathbb{R}\times[0,1]$ as in Example \ref{Ex: area preserving condition is necessay}. Then, we  get a sphere. and $g$ descends to an orientation and area preserving diffeomorphism $f$ of the sphere that has exactly two fixed points.  Note the facts that  every maximal isotopy $I$ fixes both fixed points of $f$,  that the rotation number of $I$ at each fixed point is an integer, and that the sum of the rotation numbers of $I$ at both fixed point is $1$.   So, there does not exist any maximal isotopy $I$ such that for all $z\in\mathrm{Fix}(I)$,
\[-1<\rho(I,z)<1.\]
\end{example}

\begin{example}(Example of Remark \ref{R: globle torsion-low 1})\label{Ex: local torsion low not imply globle}

 In this example, we will construct an isotopy $I^*$ on the sphere such that $I^*$ is torsion-low at a fixed point $z$, but there does not exist any torsion-low maximal isotopy  that  is equivalent to $I^*$ as a local isotopy at $z$.

We will induce the isotopy by generating functions (see Appendix \ref{S: generating fucntion}).

Let $\varphi$ be a smooth $1$-periodic  function on $\mathbb{R}$  that satisfies
\[\varphi(0)=\varphi(3/4)=\varphi(1)=0 \text{ and }|\varphi|\le\frac{1}{2\pi},\]
 \[\left\{ \begin{array}{ll}\varphi(s)>0 \text{ for }  0<s<3/4\\
 \varphi(s)<0 \text{ for } 3/4<s<1 \end{array}\right. ,\quad \text{and } \int_0^1\varphi(s)ds=0,\]
 \[|\varphi(s)|<s\sin^2\frac{\pi}{s} \quad \text{for} \quad 3/4<s<1.\]

  Let
\[ g(x,y)=\left\{\begin{array}{ll} 0 &\text{ for } y\le 0,\\
\int_0^y (s\sin^2\frac{\pi}{s}+\varphi(s)\sin ^2\pi x) ds   &\text{ for }0<y< 1,\\
\int_0^1 s\sin^2\frac{\pi}{s} ds   &\text{ for }  y\ge 1.
\end{array}\right.\]
Then, $g$ is constant on  $\mathbb{R}\times (-\infty,0]$  and on $\mathbb{R}\times [1,\infty)$ respectively, and satisfies $g(x+1,y)=g(x,y)$. Moreover, one knows that
\[ \partial^2_{12} g(x,y)=\left\{\begin{array}{ll} 0 &\text{ for } y\le 0 \text{ or } y\ge 1,\\
\pi\varphi(y)\sin(2\pi x) &\text{ for }0<y<1.
\end{array}\right.\]
So, $\partial^2_{12}g\le\frac{1}{2}<1$. Therefore, $g$ defines an identity isotopy $I=(f_t)_{t\in[0,1]}$ by the following equations:
\[f_t(x,y)=(X^t,Y^t)\Leftrightarrow\left\{\begin{aligned} X^t-x & = & t \partial_2 g(X^t,y),\\ Y^t-y & = &  -t \partial_1 g(X^t,y) , \end{aligned}\right.\]
For every $t\in[0,1]$, $f_t$ is the identity on $\mathbb{R}\times (-\infty,0]\cup \mathbb{R}\times [1,\infty)$, and satisfies $f_t(x+1,y)=f_t(x,y)+(1,0)$. Moreover, for every $t\in(0,1]$, a point $(x,y)$ is a fixed point of $f_t$ if and only if it is  a critical point of $g$. Let $\mathcal{F}$ be the foliation whose leaves are the integral curves of the gradient vector field $(x,y)\mapsto (\partial_1 g(x,y),\partial_2 g(x,y))$ of $g$. As will be proved in Appendix \ref{S: generating fucntion}, $\mathcal{F}$ is a transverse foliation of $I$.

 We know that
\[\partial_1 g(x,y)=\left\{\begin{array}{ll} 0 &\text{ for } y\le 0 \text{ or } y\ge 1,\\
\pi\sin(2\pi x)\int_0^y\varphi(s)ds &\text{ for }0<y<1,
\end{array}\right.\]
 and that
\[\partial_2 g(x,y)=\left\{\begin{array}{ll} 0 &\text{ for } y\le 0 \text{ or } y\ge 1,\\
 y\sin^2\frac{\pi}{y}+\varphi(y)\sin^2(\pi x) &\text{ for }0<y<1.
\end{array}\right.\]
So, the set of critical points of $g$ is
\[C=\{(n,\frac{1}{m}): n\in\mathbb{Z}, m\in\mathbb{N}\}\cup\mathbb{R}\times (-\infty,0]\cup\mathbb{R}\times [1,\infty),\]
and one deduces that $\partial_2 g(x,y)>0$ for $(x,y)\notin C$.

We define an equivalence relation $\sim$ on $\mathbb{R}^2$ by
\begin{eqnarray*}
\left\{\begin{array}{ll} (x,y)\sim (x',y') &\text{ for } y,y'\le 0,\\
(x,y)\sim (x+1,y)   &\text{ for }0<y< 1,\\
(x,y)\sim (x',y') &\text{ for } y,y'\ge 1.
\end{array}\right.
\end{eqnarray*}
Then, $\mathbb{R}^2/_{\sim}$ is a sphere,  $f_1$  descends to an area preserving homeomorphism $f'$ of the sphere, $I$ descends to an identity isotopy $I'$ of $f'$, and $\mathcal{F}$ descends to  a transverse foliation $\mathcal{F}'$ of $I'$. Moreover, one knows that $\mathrm{Fix}(I')=\mathrm{Fix}(f')=\mathrm{Sing}(\mathcal{F}')$, where $\mathrm{Sing}(\mathcal{F}')$ is the set of singularities of $\mathcal{F}'$. We denote by $S$  and $N$ the  two points $\mathbb{R}\times(-\infty,0]/\sim$ and $\mathbb{R}\times[1,\infty)/\sim$  in the sphere respectively.
\begin{figure}[h]
\center
   \includegraphics[width=3cm]{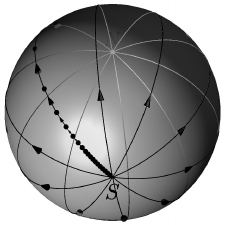}
  \caption{A sketch map of  $\mathcal{F}'$}
\end{figure}

 The fixed point $S$ is not  isolated in $\mathrm{Fix}(I')$, and
clearly $f'$ can be blown up at $S$, so $\rho_s(I',S)$ is reduced to $0$; $N$ is isolated in $\mathrm{Fix}(f')$ and is  a sink of $\mathcal{F}'$; and all the other fixed points of $f'$ are isolated in  $\mathrm{Fix}(f')$ and are saddles of $\mathcal{F}'$. Let $I^*$ be an identity isotopy of $f'$  fixing $S$ such that  $\rho_s(I^*,S)$ is reduced to $-1$. Recall Remark \ref{R: ac means torsion-low},  $I^*$ is torsion-low at $S$. We will prove that there does not exist any torsion-low maximal isotopy $I''$ such that  $\rho_s(I'',S)$ is reduced to $-1$.

 Indeed, a maximal isotopy of $f'$  fixes either  all the fixed points of $f'$ (in which case, the  isotopy is homotopic to $I'$ relatively to $\mathrm{Fix}(f')$) or exactly two fixed points. If  $I''$ is  a maximal isotopy of $f$ such that $\rho_s(I'',S)$ is reduced to $-1$, then $I''$ fixes exactly two fixed points. Denote by $\{S,z_1\}$ the set of fixed points of $I''$. One knows that $z_1$ is an isolated fixed point of $f'$, and that $J^{-1}_{z_1}I''$ is equivalent to $ I'$ as local isotopies at $z_1$. Therefore, $J^{-1}_{z_1}I''$ does not have a negative rotation type at $z_1$, and hence $I''$ is not torsion-low at $z_1$.
\end{example}

\begin{example}\label{Ex: zero not always in rotation set}(Example of Remark \ref{R: globle torsion-low 2})

 In this example, we will construct an orientation and area preserving homeomorphism $f$ of the sphere without isolated fixed points.  However,  we will show that for any maximal isotopy $I$ of $f$, $\mathrm{Fix}(I)$ contains isolated point.

Let $g$ be a homeomorphism on $\mathbb{R}\times[0,1]$ that is defined by
 \[g(x,y)=\left\{\begin{array}{ll}(x,y) &\text{ for } 0\le y\le \frac{1}{3},\\
 (x+3y-1,y)&\text{ for } \frac{1}{3}< y\le \frac{2}{3},\\
 (x+1,y)&\text{ for } \frac{2}{3}< y \le 1.\end{array}\right. \]
Like in Example \ref{Ex: area preserving condition is necessay},  we define an equivalence relation $\sim$ on $\mathbb{R}\times[0,1]$, and get a sphere as the quotient space. Moreover, $g$ descends to an orientation and area preserving diffeomorphism $f$ of the sphere that has infinitely many fixed points, and every fixed point of $f$ is not isolated in $\mathrm{Fix}(f)$. We will prove for every maximal isotopy $I$ of $f$, $\mathrm{Fix}(I)$ contains  an isolated point.

Denote by $N$ and $S$ the two components of $\mathrm{Fix}(f)$ respectively. Let us observe the properties of any maximal isotopy of $f$. Indeed, if $I$ is a maximal isotopy of $f$, it satisfies one of the following properties:
\begin{itemize}
\item[-] The set of fixed points of $I$ is the union of  $N$ (resp. $S$) and a point $z$ in $S$ (resp. $N$). In this case, $z$ is isolated in $\mathrm{Fix}(I)$;
\item[-] The set of fixed points of $I$ is the union of a point $z_1$ in $N$ and a point $z_2$ in $S$. In this case, both $z_1$ and $z_2$ are isolated in $\mathrm{Fix}(I)$;
\item[-] The set of fixed points of $I$ is a subset of  $N$ (resp. $S$) with exactly two points $z_1$ and $z_2$. In this case, both $z_1$ and $z_2$ are isolated in $\mathrm{Fix}(I)$.
\end{itemize}
In any case, $\mathrm{Fix}(I)$ contains  an isolated point.
\end{example}

\appendix

\section{Prime-ends compactification and rotation number}\label{S: prime ends compactification}

Let $f:\mathbb{R}^2\to \mathbb{R}^2$ be an orientation and area preserving homeomorphism and $U\subsetneq \mathbb{R}^2$ be a topological disk that is invariant under $f$. We will study the dynamics of $f|_{\partial U}$ near the boundary $\partial U$.

On one hand, we consider the one point campactification of $U$ and can get a homeomorphism of a sphere. To avoid too many  notations, we will still denote the point added by $\partial U$. On the other hand, by Carath\'eodory's prime-ends theory (see \cite{Lecalvezprimeendsrotationnumberandperiodicpoints} for example), we can compactify $U$ by adding a circle $S^1$ on the boundary and extend $f|_{U}$ continuously to the boundary. So we get a blow-up of $f|_{U}$ at the point (the end) $\partial U$. We define the \emph{prime-end rotation number} $\rho(f,\partial U)$  of $f$ at $\partial U$ to be the Poincar\'e's rotation number ($\in\mathbb{R}/\mathbb{Z}$) of the homeomophism on the boundary $S^1$. In particular, if $\partial U\subset \mathrm{Fix}(f)$, then $\rho(f,\partial U)=0\in\mathbb{R}/\mathbb{Z}$ (\cite[Propositions 3.6  and Proposition 5.7 ]{Lecalvezprimeendsrotationnumberandperiodicpoints}).

Moreover, let $I=(f_t)_{t\in[0,1]}$ be an isotopy from the identity to $f$. We will prove the following result:
\begin{proposition}\label{P: zero prime end rotation number}
If $I$ fixes every point of $\partial U$ and a point $z_0$ in $U$, then
\begin{itemize}
\item[i)] for all $t\in[0,1]$, the extension  of $f_t$ to $S^1$ is the identity map;
\item[ii)] the blow-up rotation number $\rho(I,\partial U)=0\in\mathbb{R}$.
\end{itemize}
\end{proposition}

\begin{remark}
The condition that $I$ fixes a point in $U$ is to make sure that it is a local isotopy at $\partial U$.
\end{remark}

We will mainly follow the idea in \cite{Lecalvezprimeendsrotationnumberandperiodicpoints}. Before proving the proposition, we will recall some definitions and results (also in \cite{Lecalvezprimeendsrotationnumberandperiodicpoints}).

An \emph{end-cut} of $U$ is an arc $\gamma: [0,1)\to U$ such that $\lim_{s\to 1-}\gamma(s)$ is a point in $\partial U$. A point $z\in\partial U$ is \emph{accessible} (from $U$) if it is the endpoint of some end-cut in $U$. Note that accessible points are dense in $\partial U$.
A \emph{cross-cut} of $U$  is  a simple arc $\gamma:(0,1)\to U$  joining two points of $\partial U$ such that each of the  two components of $U\setminus \gamma$  has a boundary point in $\partial U$ different from the endpoints of $\gamma$.
 A \emph{cross-section} of $U$ is any connected component of $U\setminus \gamma$ for some cross-cut $\gamma$.
A \emph{chain} for $U$  is a sequence $(D_n)_n$ of cross-sections such that $D_i\subset D_j$ for all $i\ge j$ and $\partial_U D_i\cap \partial_U D_j=\emptyset$ for all $i\ne j$. We say that
a chain $(D_n)_n$ \emph{divides} a cross-section $D$ if $D_i\subset D$ for all sufficiently large $i$. We say that a chain $(D_n)_n$ \emph{divides} a chain $(D'_n)_n$ if  $(D_n)_n$ divides $D'_m$ for all $m$.  Two chains $(D_n)_n$ and $(D'_n)_n$ are \emph{equivalent} if  $(D_n)_n$ divides $(D'_n)_n$ and $(D'_n)_n$ divides $(D_n)_n$. A chain $(D_n)_n$ is \emph{prime} if it divides $(D'_n)_n$ whenever $(D'_n)_n$ is a chain
that divides it.    An equivalence class of prime chains is called a \emph{prime-end} of $U$. We say that a prime-end divides a cross-section $D$ if one of its representative divides $D$.

We topologize the set of all prime-ends of $U$, together with $U$, by defining a basis of open sets consists of all the set with the form $D\sqcup\{\text{the prime-ends divide $D$}\}$ for some cross-section $D$, together with the open subset of $U$. With the topology, the set of all prime-ends of $U$, together with $U$, is homeomorphic to a closed disk. We will denote it by $U\sqcup S^1$.
We say that a prime-end $p\in S^1$ is \emph{accessible} if there is an end-cut $\gamma$ such that $\gamma(t)\to p$ in $U\sqcup S^1$ as $t\to 1-$. Note that accessible prime-ends are dense in $S^1$.

\begin{proof}[Proof of the first part of Proposition \ref{P: zero prime end rotation number}] We will prove the result for $f$ (for $f_t$, the proof is similar).  We use the same notation $f$ for the extension of $f|_{U}$ to $U\sqcup S^1$.

We will prove the result by contradiction. Suppose that  $f|_{S^1}$ is not the identity, then there exists an interval  $I\subset S^1$ such that $f(I)\cap I=\emptyset$. By the density of accessible prime-ends, we can choose $p\in I$ and an end-cut $\gamma$ such that the end of $\gamma$ in $S^1$ is $p$. We denote by $z_1$ the end of $\gamma$ in $\partial U$.

Because $\rho(f,\partial U)=0\in\mathbb{R}/\mathbb{Z}$ (Proposition 5.7 of \cite{Lecalvezprimeendsrotationnumberandperiodicpoints}), $f^2(p)\ne p$. By choosing $\gamma$ ``short" enough, we can suppose that $\gamma$, $f(\gamma)$ and $f^2(\gamma)$ are pairwise disjoint. Let $I_1\subset S^1$ be the closed interval bounded by $p$ and $f(p)$ such that $f^2(p)\notin I_1$. Then $I_1\cap f(I_1)=\{f(p)\}$. We join $\gamma(0)$ and $f(\gamma(0))$ by a simple arc $\eta$ and get a simple arc $\sigma=\gamma\cup \eta\cup f(\gamma)$.  Since $\sigma$ extends in $S^1$ to two different points $p$ and $f(p)$, it separates $U$ to two open topological disks. We denote by $D_1$ the one bounded by $I_1\cup \sigma$ in $U\sqcup S^1$, and by $D_2$ the other one (see Figure \ref{fig: prime-ends}).
Recall that $f(I_1)\cap I_1=\{f(p)\}$. So, for each ``short" enough end-cut $\gamma'$ in $D_1$, $f(\gamma')\subset D_2$.

 On the other hand, we consider $D_1\subset U\subset \mathbb{R}^2$. Recall that both ends of $\sigma$ extends in $\partial U\subset \mathbb{R}^2$ to $z_1$. So, $\sigma\cup\{z_1\}$ is a simple loop in $\mathbb{R}^2$, and separates $\mathbb{R}^2$ to two components. We denote by $C_1$ the component containing $D_1$, and by $C_2$ the component containing $D_2$.

 We will deduce that $C_1\cap \partial U=\emptyset$, which means $C_1=D_1$.  In fact, if  $C_1\cap \partial U\ne \emptyset$, by the density of accessible points, there exists an end-cut $\gamma'\subset D_1\subset C_1$ whose end in $\partial U$ is a point $z_2\ne z_1$. By choosing $\gamma'$ ``short" enough, we get $f(\gamma')\subset D_1$.   we get a contradiction.

 \begin{figure}[ht]
 \subfigure[On $U\sqcup S^1$]{
 \begin{minipage}[b]{0.35\linewidth}
   \center
  \includegraphics[width=0.6\linewidth]{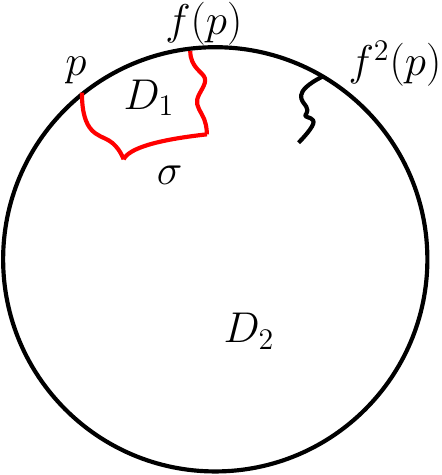}
 \end{minipage}
 }
 \hfill
 \subfigure[On $\mathbb{R}^2$]{
 \begin{minipage}[b]{0.45\linewidth}
   \center
  \includegraphics[width=0.6\linewidth]{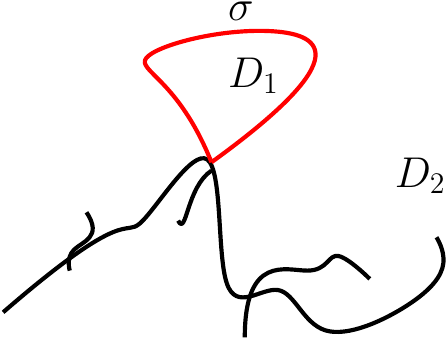}
 \end{minipage}
 }
 \caption{}\label{fig: prime-ends}
 \end{figure}

Now, we know that all the end-cuts in $D_1$ extend to $z_1$ in $\partial U$. It means that there is no cross-section in $D_1$. Recall the topology of $U\sqcup S^1$, we get a contradiction.
\end{proof}

\begin{proof}[Poof of the second part of Proposition \ref{P: zero prime end rotation number}] Let us consider a universal cover $\pi:\mathbb{R}^2_{-}\to U\sqcup S^1\setminus \{z_0\}$ and the natural lift $(\widetilde{f}_t)_{t\in[0,1]}$ of $(f_t)_{t\in[0,1]}$ to $\mathbb{R}^2_{-}$. We know that $f_t|_{S^1}$ is the identity and $\widetilde{f}_0$ is the identity. We will prove that for each $k\in\mathbb{Z}$, $\{t\in[0,1]:\widetilde{f}_t(x,0)=(x+k,0)\}$ is an open subset of $[0,1]$, and hence for all $t\in[0,1]$, the restriction of $\widetilde{f_t}$ to $\mathbb{R}\times\{0\}$ is the identity. It implies $\rho(I,\partial U)=0\in\mathbb{R}$.

Fix $t_0\in[0,1]$, and suppose that $\widetilde{f}_{t_0}(x,0)=(x+k,0)$.

Choose two different accessible points $z_1$, $z_2 \in \partial U$, and an end-cut $\gamma_1\subset U$ that extends to $z_1\in\partial U$. We can extend $\gamma_1$ and get an arc in $U$ that joining $z_0$ to $z_1$. We still denote it by $\gamma_1$. By choosing $\gamma_1$ properly, we can suppose that it is a simple arc. Since $z_2$ is a fixed point of $f_{t_0}$, there exist $\delta>0$ and  a neighborhood $V$ of $z_2\in\mathbb{R}^2$ such that for all $ |t-t_0|<\delta $, $f_t(V)\cap\{\gamma_1\}=\emptyset$. Choose an end-cut $\gamma_2\subset V$ that extends to $z_2\in\partial U$.

 Denote by $p_1\in S^1$ one end point of $\gamma_1$. The lifts of $\gamma_1\cup\{p_1\}$ to $\mathbb{R}^2_{-}$  are infinitely many arcs from $\mathbb{R}\times\{0\}$ to the infinity and separate $\mathbb{R}^2_{-}$ into infinitely many components. We choose one component and denote it by $\widetilde{C}_0$. We denote by $\widetilde{C}_k=\widetilde{C}_0+(k,0)$, which is still a component of $\pi^{-1}(U\sqcup S^1\setminus (\gamma_1\cup\{p_1,z_0\}))$.

 Denote by $p_2\in S^1$ the end point of $\gamma_2$, and $\widetilde{p}_2\in \widetilde{C}_0$ one lift of $p_2$. Let $\widetilde{\gamma}_2\subset \widetilde{C}_0$ be one lift of $\gamma_2$. We know that $\widetilde{f}_{t_0}(\widetilde{p}_2)=\widetilde{p}_2+(k,0)$, and hence $\widetilde{f}_{t_0}(\widetilde{\gamma}_2)\subset \widetilde{C}_k$. Recall that $f_t(V)\cap\{\gamma_1\}=\emptyset$ and $\gamma_2\subset V$. By the continuity of $(\widetilde{f}_t|_{U})_{t\in[0,1]}$ with respect to $t$, we know that for $|t-t_0|<\delta$, $\widetilde{f}_{t}(\widetilde{\gamma}_2)\subset \widetilde{C}_k$ and hence $\widetilde{f}_t(\widetilde{p}_2)=\widetilde{p}_2+(k,0)$.
Noting that $f_t$ is the identity map on $S^1$, we must have $\widetilde{f}_t(x,0)=(x+k,0)$ for all $x\in\mathbb{R}$.
\end{proof}

\section{Construction of a transverse foliation from the generating function}\label{S: generating fucntion}

Let $f$ be a diffeomorphism of $\mathbb{R}^2$ and $g: \mathbb{R}^2\rightarrow \mathbb{R}$ be a $\mathcal{C}^2$ function. We call $g$ a \emph{generating function}\footnote{Our definition of generating function is indeed the same as the one in Section 9.2 of \cite{Mcduff}. We can get all the statements in this paragraph by repeating their proofs.} of $f$ if  $\partial^2_{12}g<1$, and if
\[ f(x,y)=(X,Y)\Leftrightarrow\left\{\begin{aligned} X-x & = & \partial_2 g(X,y),\\ Y-y & = &  -\partial_1 g(X,y) . \end{aligned}\right.\]
Every $\mathcal{C}^2$ function $g: \mathbb{R}^2\rightarrow \mathbb{R}$ satisfiying $\partial^2_{12} g \le c<1$ defines a diffeomorphism $f$ of $\mathbb{R}^2$ by the previous equations; on the other side, for every area preserving diffeomorphism $f$ of $\mathbb{R}^2$ satisfying $0<\varepsilon\le\partial_1(p_1\circ f)\le M<\infty $, where $p_1$ is the projection onto the first factor, there exists a generating function of $f$. Moreover,  the Jacobian matrix $J_{f}$ of $f$ is  equal to
\[\frac{1}{1-\partial^2_{12}g(X,y)}
\begin{pmatrix} 1 &    \partial^2_{22}g(X,y)\\
-\partial^2_{11}g(X,y)   & -\partial^2_{11}g(X,y) \partial^2_{22}g(X,y) +(1-\partial^2_{12} g(X,y))^2
\end{pmatrix}.\]
Since  $\det J_f=1$, the diffeomorphism  $f$  is orientation and area preserving.  A point $(x,y)$ is a fixed point of $f$ if and only if it is a critical point of $g$.  We can naturally define an identity isotopy $I=(f_t)_{t\in[0,1]}$ of $f$ such that $f_t$ is generated by $tg$. Precisely, the diffeomorphisms $f_t$ are defined by the following equations:
\[ f_t(x,y)=(X^t,Y^t)\Leftrightarrow\left\{\begin{aligned} X^t-x & = & t \partial_2 g(X^t,y),\\ Y^t-y & = &  -t \partial_1 g(X^t,y) . \end{aligned}\right.\]

In this section, we suppose that $f$ is a diffeomorphism of $\mathbb{R}^2$, and that $g$ is a generating function of $f$. We will construct a transverse foliation of $I$. More precisely, denote by  $\mathcal{F}$ the foliation  whose leaves are the integral curves of the gradient vector field $(x,y)\mapsto(\partial_1 g(x,y), \partial_2 g(x,y))$ of $g$, we will prove the following result:
\begin{theorem}
 The foliation  $\mathcal{F}$ is a transverse foliation of $I$.
\end{theorem}

\begin{proof}
We will prove the theorem by constructing an identity isotopy $I'$ of $f$ that is homotopic to $I$ relatively to $\mathrm{Fix}(f)$ and satisfies that for every $z\in\mathbb{R}^2\setminus\mathrm{Fix}(f)$, the trajectory of $z$ along $I'$ is positively transverse to $\mathcal{F}$.

We  define  $I'=(f'_t)_{t\in[0,1]}$ by the following equations:
\[f'_t(x,y)=\left\{\begin{aligned}(x, y)+2t(X-x,0) & & \textrm{for } 0\leq t\leq 1/2 ,\\ (X,y)+(2t-1)(0, Y-y) & & \textrm{for } 1/2\leq t\leq 1,\end{aligned}\right.\]
 where $(X,Y)=f(x,y)$.
 \begin{lemma}
 One can verify that $I'$ is  an identity isotopy of $f$.
 \end{lemma}
 \begin{proof}
 We know that $\partial_1 X(x,y)=1/(1-\partial^2_{12}g(X,y))>0$. By  computing the determinant of the Jacobian matrix of $f'_t$, we know that $\det J_{f'_t}>0$ for every $t\in[0,1]$.
To prove that $I'$ is an isotopy, we only need to check that $f'_t$ is a bijection for every $t\in(0,1)$.

For $t\in(0,\frac{1}{2})$,
write $f'_t(x,y)=(\varphi_{t,y}(x),y)$.  One deduces
\[\frac{\partial}{\partial x}\varphi_{t,y}(x)=2t\partial_1 X(x,y)+(1-2t)\ge 1-2t> 0.\]
So, $f'_t$ is a bijection.

For $t=\frac{1}{2}$, $f'_{1/2}(x,y)=(X,y)$. We have $\partial_1 X(x,y)>0$, so $f'_{1/2}$ is an injection. Recall that $X-x=\partial_2 g(X,y)$, so $x=X-\partial_2 g(X,y)$, $f'_{1/2}$ is a surjection.

For $t\in(\frac{1}{2},1)$, write $f'_t(x,y)=(X,\psi_{t,X}(y))$.
One deduces
 \[\frac{\partial}{\partial y}\psi_{t,X}(y)=1-(2t-1)\partial^2_{12}g(X,y)>2-2t>0.\]
 So, $(X,y)\mapsto (X,\psi_{t,X}(y))$ is a bijection, and hence $f'_t$ is a bijection.
 \end{proof}

 By definition, we know $\mathrm{Fix}(I)=\mathrm{Fix}(I')=\mathrm{Fix}(f)$. If $\mathrm{Fix}(f)$ is empty or contains more than one point, we know that $(\mathrm{Fix}(f),I)\sim (\mathrm{Fix}(f),I')$, and hence a transverse foliation of $I'$ is also a transverse foliation of $I$; if $\mathrm{Fix}(f)$ is reduced to one point, we can  deduce the same result by the following lemma and the fact that $\pi_1(\mathrm{homeo}_0(\mathbb{R}^2,0))\cong\mathbb{Z}$.

\begin{lemma}
If $0$ is an isolated fixed point of $f$, one can deduce that $\rho(I,0)=\rho(I',0)\in[-1,1]$.
\end{lemma}

\begin{proof}
Let $\theta:[0,1]\rightarrow \mathbb{R}$ and $\theta':[0,1]\rightarrow \mathbb{R}$  be the continuous functions that satisfies
$\theta(0)=\theta'(0)=0$ and
\[\frac{J_{f_t}(0)\begin{pmatrix}1\\ 0\end{pmatrix}}{\|J_{f_t}(0)\begin{pmatrix}1\\ 0\end{pmatrix}\|}=\begin{pmatrix}\cos\theta(t)\\ \sin\theta(t)\end{pmatrix}, \quad
 \frac{J_{f'_t}(0)\begin{pmatrix}1\\ 0\end{pmatrix}}{\|J_{f'_t}(0)\begin{pmatrix}1\\ 0\end{pmatrix}\|}=\begin{pmatrix}\cos\theta'(t)\\ \sin\theta'(t)\end{pmatrix}.\]
 To simplify the notations, we write
 \[\mathrm{Hess}(g)(0)=\begin{pmatrix}\varrho, \sigma\\ \sigma,\tau\end{pmatrix}.\]
One knows
 \[J_{f_t}(0)\begin{pmatrix}1\\ 0\end{pmatrix}=\frac{1}{1-t\sigma}\begin{pmatrix}1\\ -t\varrho\end{pmatrix}.\]
Because $1-t\sigma>0$ for all $t\in[0,1]$, we deduce that $\theta(t)$ belongs to $(-\frac{\pi}{2},\frac{\pi}{2})$ for all $t\in[0,1]$.

 For $t\in[0,\frac{1}{2}]$,
 \[J_{f'_t}(0)\begin{pmatrix}1\\ 0\end{pmatrix}=\begin{pmatrix}(1-2t)+2t\partial_1 X(0,0)\\ 0\end{pmatrix},\]
and $(1-2t)+2t\partial_1 X(0,0)>0$. So $\theta'(t)$ is equal to $0$ for all $t\in[0,\frac{1}{2}]$.

  For $t\in[\frac{1}{2},1]$,
 \[J_{f'_t}(0)\begin{pmatrix}1\\ 0\end{pmatrix}=\begin{pmatrix}\partial_1 X(0,0)\\ (2t-1)\partial _1 Y(0,0)\end{pmatrix},\]
and $\partial_1 X(0,0)>0$. So $\theta'(t)$ belongs to $(-\frac{\pi}{2},\frac{\pi}{2})$ for all $t\in[\frac{1}{2},1]$.

 Therefore, we deduce that $\theta(1)=\theta'(1)\in (-\frac{\pi}{2},\frac{\pi}{2})$, and hence $\rho(I,0)=\rho(I',0)\in[-1,1]$.
\end{proof}

Next, we show that  $I'$ intersects $\mathcal{F}$ positively transverse.
\begin{lemma}\label{L: construct a transverse foliation}
For every $z=(x,y)$ that is not a fixed point of $f$,  the path $\gamma_z:t\mapsto f'_t(x,y)$ is   positively transverse to $\mathcal{F}$.
\end{lemma}

\begin{figure}[H]
\center
   \includegraphics[width=3.5cm]{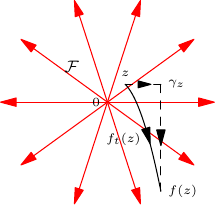}
  \caption{The dynamics and foliation generated by $g(x,y)=x^2+y^2$}
\end{figure}

\begin{proof}[Proof of Lemma \ref{L: construct a transverse foliation}]
To prove the positive transversality, the rough idea is to calculate the determinate of the matrix formed by two vectors tangent respectively to the curve $\gamma_z$ and to the leaf at the intersecting point.

For $t\in[0,1/2]$,
\begin{align*}
&\det\begin{pmatrix}
2(X-x) & \partial_1 g(f'_t(x,y))\\
0 &\partial_2 g(f'_t(x,y))
\end{pmatrix}\\
=&2(X-x)\partial_2 g(f'_t(x,y))\\
=&2(X-x)\partial_2 g(2tX+(1-2t)x,y)\\
=&2(X-x)[\partial_2 g(X,y)+(2t-1)(X-x)\partial^2_{12}g(\xi, y)]  \\
=&2(X-x)^2[1-(1-2t)\partial^2_{12}g(\xi, y)]\ge 0,
\end{align*}
where $\xi$ is a real number between $x$ and $X$, and the inequality is strict if $X\ne x$.

 For $t\in[1/2,1]$, similarly, we have
\[\det\begin{pmatrix}
0 & \partial_1 g(f'_t(x,y))\\
2(Y-y) &\partial_2 g(f'_t(x,y))
\end{pmatrix}
=2(Y-y)^2[1-(2t-1)\partial^2_{12} g (X ,\eta)]\ge 0,\]
where $\eta$ is a real number between $y$ and $Y$, and the inequality is strict if $Y\ne y$.

Since $z$ is not a fixed point, either $X\ne x$ or $Y\ne y$. If both  inequalities are satisfied, $\gamma_z$ intersects $\mathcal{F}$ positively transversely; if $X\ne x$ and $Y= y$, $\gamma_z|_{t\in[0,\frac{1}{2}]}$  intersects $\mathcal{F}$ positively transversely, and $\gamma_z|_{t\in[\frac{1}{2},1]}$ is reduced to a point; if $X= x$ and $Y\ne y$, $\gamma_z|_{t\in[0,\frac{1}{2}]}$ is reduced to a point, and $\gamma_z|_{t\in[\frac{1}{2}, 1]}$ intersects $\mathcal{F}$ positively transversely.
\end{proof}
Since $I$ and $I'$ are homotopic relatively to $\mathrm{Fix}(f)=\mathrm{Fix}(I)=\mathrm{Fix}(I')$, and $\mathcal{F}$ is a transverse foliation of $I'$, $\mathcal{F}$ is also a transverse foliation of $I$.
\end{proof}

\section*{Acknowledgements}

The author wish to thank Patrice Le Calvez for  proposing me the subject, and thank Fran\c{c}ois B\'eguin, Sylvain Crovisier, and Fr\'ed\'eric Le Roux  for explaining to me their  results  and especially Le Roux for some valuable comments.

\end{document}